\documentclass[11pt]{amsart}
\usepackage{amsmath}
\usepackage{amssymb}
\usepackage{amscd}
\usepackage{color}

\def\NZQ{\mathbb}               
\def\NN{{\NZQ N}}
\def\QQ{{\NZQ Q}}
\def\ZZ{{\NZQ Z}}
\def\RR{{\NZQ R}}

\newtheorem{Theorem}{Theorem}[section]
\newtheorem{Lemma}[Theorem]{Lemma}
\newtheorem{Corollary}[Theorem]{Corollary}
\newtheorem{Proposition}[Theorem]{Proposition}
\newtheorem{Remark}[Theorem]{Remark}

\newtheorem{Example}[Theorem]{Example}

\newtheorem{Definition}[Theorem]{Definition}

\let\epsilon\varepsilon
\let\phi=\varphi
\let\kappa=\varkappa

\begin{document}

\title{Multiplicities of graded families of ideals on Noetherian local rings }
\author{Steven Dale Cutkosky}

\thanks{Partially supported by NSF grant DMS-2348849.}

\address{Steven Dale Cutkosky, Department of Mathematics,
University of Missouri, Columbia, MO 65211, USA}
\email{cutkoskys@missouri.edu}

\begin{abstract}
Let $R$ be a $d$-dimensional Noetherian local ring with maximal ideal $m_R$.
In this article, we give a generalization of the multiplicity $e(I)$ of an $m_R$-primary ideal $I$ of $R$ to a multiplicity $e(\mathcal I)$ of a graded family of $m_R$-primary ideals $\mathcal I$ in $R$. This multiplicity gives the classical multiplicity $e(I)$ if $\mathcal I=\{I^n\}$ is the $I$-adic filtration, and agrees with the volume, $\displaystyle \lim_{n\rightarrow \infty}d!\frac{\ell(R/I_n) }{n^d}$ for $R$ such  that the volume always exists as a limit. We will show in this paper that many of the classical theorems for the multiplicity of an ideal generalize to this multiplicity, including mixed multiplicities, the Rees theorem and the Minkowski inequality and equality. We give simple proofs which are independent of the theory of volumes and Okounkov bodies for all of our results, with the one exception being  the proof of the Minkowski equality. We do this by interpreting the multiplicity of graded families of $m_R$-primary ideals as a limit of  intersection products on the family of $R$-schemes which are obtained by blowing up $m_R$-primary ideals in $R$.

\end{abstract}


\maketitle

\section{Introduction}

Throughout this article, $R$ will be a Noetherian local ring of dimension $d$ with maximal ideal $m_R$. The length of an $R$-module $M$ will be denoted by $\ell(M)$ or $\ell_R(M)$. If $I$ is an $m_R$-primary ideal of $R$, then the multiplicity of $R$ with respect to $I$ is 
$$
e(I)=\lim_{n\rightarrow\infty}d!\frac{\ell(R/I^n)}{n^d}.
$$
A graded $m_R$-family $\mathcal I=\{I_n\}_{n\in\ZZ_{\ge 0}}$ on  $R$ is a family of ideals such that 
$I_0=R$ and  $I_m$ is $m_R$-primary for $m>0$ such that $I_mI_n\subset I_{m+n}$ for all $m,n\ge 0$. The Rees algebra 
$$
R[\mathcal I]=\oplus_{n\ge 0}I_n
$$
 is then naturally a graded $R$-algebra, which is generally not finitely generated. $\mathcal I$ is a  graded $m_R$-filtration if we also have the condition that 
$I_{n+1}\subset I_n$ for all $n$. 

 The volume 
$\mbox{vol}(\mathcal I)$ of a graded $m_R$-family $\mathcal I$ is defined as
\begin{equation}\label{eq34}
\mbox{vol}(\mathcal I)=\limsup_{n\rightarrow\infty} d!\frac{\ell(R/I_n)}{n^d}.
\end{equation}
This limsup is always a nonnegative real number (since there exists a positive integer $s$ such that $m_R^{sn}\subset I_n$ for all $n$).

The problem of existence of such limsups (\ref{eq34}) as limits  has been considered by Ein, Lazarsfeld and Smith \cite{ELS} and Musta\c{t}\u{a} \cite{Mus}.
When the ring $R$ is a domain and is essentially of finite type over an algebraically closed field $k$ with $R/m_R=k$, Lazarsfeld and Musta\c{t}\u{a} \cite{LM} showed that
the limit exists for all graded $m_R$-families.  We have the following general theorem.
\begin{Theorem}(Theorem 1.1 \cite{C2})\label{TheoremVolEx}
Suppose that $N(\hat R)$ is the nilradical of the $m_R$-adic completion $\hat R$ of $R$. Then the limit
\begin{equation}\label{eq10}
\lim_{n\rightarrow\infty} d!\frac{\ell_R(R/I_n)}{n^d}
\end{equation}
exists for any graded $m_R$-family $\mathcal I=\{I_n\}$ on $R$ if and only if $\dim N(\hat R) < d$.
\end{Theorem}
 The condition $\dim N(\hat R)<d$ holds for instance if $R$ is analytically unramified ($\hat R$ is reduced). It was shown by Hailong Dao and Ilya Smirnov that if $\dim N(\hat R)=d$ then there exists a graded $m_R$-family $\mathcal I$ on $R$ such that $\mbox{vol}(\mathcal I)$ is not a limit.
 
 The following ``Volume = Multiplicity'' formula is proven in the generality stated below in \cite{C13}. It was earlier proven  in some cases in \cite{ELS} and  when the ring $R$ is a domain that is essentially of finite type over an algebraically closed field $k$ with $R/m_R=k$ by   Lazarsfeld and Musta\c{t}\u{a} in \cite{LM}. 
  
 \begin{Theorem}\label{volmult}(Theorem 1.1 \cite{C13})  Suppose that $\dim N(\hat R)<d$ and $\mathcal I=\{I_n\}$ is a graded $m_R$-family on $R$. Then 
\begin{equation}\label{eq33}
\mbox{\rm vol}(\mathcal I)=\lim_{p\rightarrow\infty}\frac{e(I_p)}{p^d}.
\end{equation}
\end{Theorem}
Let $R$ be a Noetherian local ring and $\mathcal I$ be a graded $m_R$-family. We define the multiplicity $e(\mathcal I)$ of 
$\mathcal I$ as
$$
e(\mathcal I)=\limsup_{p\rightarrow\infty}\frac{e(I_p)}{p^d}.
$$

Although the volume $\mbox{vol}(\mathcal I)$ does not  exist in general as a limit on an arbitrary Noetherian local ring $R$,  the multiplicity $e(\mathcal I)$ always does exist. 

\begin{Theorem}\label{multlim} Let $R$ be a $d$-dimensional Noetherian local ring and $\mathcal I=\{I_n\}$ be a graded $m_R$-family. Then the multiplicity 
$$
e(\mathcal I)=\lim_{p\rightarrow\infty}\frac{e(I_p)}{p^d}
$$
 exists as a limit. 
\end{Theorem}

Here is a simple proof which relies on the volume = multiplicity formula (\ref{eq33}) and the existence of the volume as a limit (\ref{eq10}) on complete local domains. 

\begin{proof} By the Associativity formula for  multiplicity of a module with respect to  an ideal (Theorem 11.2.4 \cite{HS}), for all $n$,
$$
e(I_n)=e(I_n\hat R)=\sum \ell_{\hat R/P}(\hat R/P)e(I_n\hat R/P),
$$
where the sum is over all minimal primes $P$ of $\hat R$ such that $\dim \hat R/P=d$.
Thus 
$$
\lim_{n\rightarrow\infty}\frac{e(I_n)}{n^d}=\lim_{n\rightarrow\infty}\frac{\left(\sum \ell_{\hat R_P}(\hat R_P)e(I_n\hat R/P)\right)}{n^d}
=\sum \ell_{\hat R/P}(\hat R/P)\lim_{n\rightarrow\infty}\frac{e(I_n\hat R/P)}{n^d}
$$
 exists by Theorem \ref{volmult} since each $\hat R/P$ is analytically irreducible, and hence $N(\hat R/P)=0$.
\end{proof}

The proofs of (\ref{eq10}) and (\ref{eq33}) use the rather involved theory of Okounkov bodies,  which was introduced in the papers \cite{Ok}, \cite{KK} and \cite{LM}.
We give  a simple direct proof of Theorem \ref{multlim} in Theorem \ref{Theorem5}, using limits of intersection products. This proof  does not use the theories of Volume or of Okounkov bodies.

In the case that $d=0$, all $m_R$-primary ideals are nilpotent, so that $e(I_n)=e(R)$ for all $n$. Thus $e(\mathcal I)=e(R)$ for all graded $m_R$-families $\mathcal I$. All theorems in this paper about multiplicities of graded $m_R$-families are thus trivial in the case that $d=0$, so we will only address the case $d\ge 1$ in our proofs.

The beginning point of this proof is the following interpretation of the multiplicity of an $m_R$-primary ideal as an intersection product.

\begin{Theorem}\label{Theorem20}(Theorem A, \cite{CM})
Let $R$ be a $d$-dimensional Noetherian local ring  and $I\subset R$ be an $m_R$-primary ideal. 
 Let $\pi:Y\rightarrow \mbox{Spec}(R)$ be a birational proper morphism such that $I\mathcal O_Y$ is invertible. Then 
\begin{equation}
e(I)=-((I\mathcal O_Y)^d).
\end{equation}
\end{Theorem}

 In \cite{Rm} Ramanujam obtained the above formula when there is a good intersection product defined on birational proper morphisms to $\mbox{Spec}(R)$. In \cite{Laz1}, the formula is stated for the
case that R is the local ring of a closed point on an algebraic scheme over an algebraically closed
field and in \cite{BLQ} the formula is proven when R is an analytically irreducible domain and Y is normal. The fact that the good intersection product exists above arbitrary Noetherian local rings $R$ follows from 
the theory of rational equivalence developed in \cite{Th} as applied to the intersection theory developed in \cite{Fl} and \cite{Kl}. 

Here is a short outline of our proof of Theorem \ref{multlim}.
Let $Y_n$ be the blowup of the ideal $I_n$ above $\mbox{Spec}(R)$. Then $I_n\mathcal O_{Y_n}=\mathcal O_{Y_n}(-
F_n)$ for some effective Cartier divisor $F_n$ on $Y_n$ with exceptional support. We compute $e(I_n)=-((-F_n)^d)$, the self intersection of $-F_n$ on $Y_n$.
Using only the fact that  $\mathcal I$ is a graded $m_R$-family, we show as a consequence of  the more general Theorem \ref{Theorem70}, that 
\begin{equation}\label{eqIn11}
\lim_{n\rightarrow\infty}\frac{e(I_n)}{n^d}=\lim_{n\rightarrow \infty}-\left(\left(\frac{-F_n}{n}\right)^d\right)
\end{equation}
exists, establishing Theorem \ref{multlim}.

We also obtain multiplicities of modules, and show that the multiplicity is a natural limit. 
If $R\rightarrow S$ is a local homomorphism and $\mathcal I=\{I_n\}$ is a graded $m_R$-family, then we define a graded $m_S$-family by 
\begin{equation}\label{eqIn5}
\mathcal IS=\{I_nS\}.
\end{equation}

 Let $M$ be a finitely generated $R$-module. Define the multiplicity
$$
e(\mathcal I;M)=\limsup_{n\rightarrow\infty}\frac{e(I_n;M)}{n^d},
$$
where $e(I_n;M)=\displaystyle\lim_{m\rightarrow\infty}d!\frac{\ell(M/I_n^mM)}{m^d}$ is the ordinary multiplicity of the Module $M$ with respect to $I_n$.

In Corollary \ref{Cor73}, we  find an associativity formula for multiplicity of graded $m_R$-families. This corollary shows 
that 
$$
\lim_{n\rightarrow\infty}\frac{e(I_n;M)}{n^d}=\sum_{P\in\Lambda}e(\mathcal I(R/P))\ell_{R_P}(M_P)
$$
where  $\Lambda$ is the set of minimal prime ideals $P$ of $R$ such that $\dim R/P=d$. 
In particular, the multiplicity
$$
e(\mathcal I;M)=\lim_{n\rightarrow\infty}\frac{e(I_n;M)}{n^d}
$$
 exists as a limit.
 
 If $\mathcal I$ is a graded $m_R$-family on a $d$-dimensional local ring $R$ and $M$ is a finitely generated $R$-module, we define
 $$
 {\rm vol}(\mathcal I;M)=\limsup_{n\rightarrow\infty}d!\frac{\ell(M/I_nM)}{n^d}.
 $$
 
 We always have that
 $$
 {\rm vol}(\mathcal I;M)\le e(\mathcal I;M)
 $$
 as is shown in Proposition \ref{PropVUB}. However, there exist examples where the inequality is strict, as is shown in Example \ref{VolEx}.
 In particular, the  volume = multiplicity formula of Theorem \ref{volmult} for rings $R$ such that $\dim N(\hat R)<d$ does not extend to more general rings.

 \subsection{Mixed multiplicities of graded $m_R$-families}\label{SecMix} 
 The theory of mixed multiplicities of $m_R$-primary ideals in a Noetherian local ring $R$
with maximal ideal $m_R$, was initiated by Bhattacharya \cite{Bh}, Rees \cite{R} and Teissier and Risler
\cite{T1}. 
 
Let $R$ be a $d$-dimensional Noetherian local ring  and   $I_1,\ldots, I_r$ be $m_R$-primary ideals of $R$. By Theorem 17.4.2 and  Definition 17.4.3 \cite{HS}, 
for any $n_1,\ldots, n_r\in \ZZ_{\ge 0}$, we have
$$
e(I_1^{n_1}\cdots I_r^{n_r}) = \sum_{\stackrel{d_1,\ldots, d_r\in \ZZ_{\ge 0},}{d_1+\cdots+d_r=d}}\frac{d!}{d_1!\cdots, d_r!}e(I_1^{[d_1]},\ldots, I_r^{[d_r]})n_1^{d_1}\cdots n_r^{d_r}
$$
is a homogeneous polynomial of degree $d$.
 The coefficients $e(I_1^{[d_1]},\ldots, I_r^{[d_r]})$ are called the mixed multiplicities with respect to $I_1,\ldots,I_r$.

The following theorem gives an interpretation of  mixed multiplicities of ideals in terms of intersection products. This formulation is derived from 
Theorem \ref{Theorem20}.

\begin{Theorem}\label{MixMultIdeal}(Theorem 5.2 \cite{CM})
Let $R$ be a $d$-dimensional Noetherian local ring  and   $I_1,\ldots, I_r$ be $m_R$-primary ideals of $R$. Let  $\pi:Y\rightarrow \mbox{Spec}(R)$ be a proper birational morphism such that $I_i\mathcal O_Y$ is locally principal for every $i$. 
 Then for every $d_1,\ldots, d_r\in \ZZ_{\ge 0}$ such that $d_1+\cdots+d_r=d$ we have
$$ 
e(I_1^{[d_1]},\ldots, I_r^{[d_r]})=-\big((I_1\mathcal O_Y)^{d_1}\cdot\ldots\cdot(I_r\mathcal O_Y)^{d_r}\big).
$$
For all $n_1,\ldots,n_r\in\ZZ_{\ge 0}$,
$$
\begin{array}{lll}
e(I_1^{n_1}\cdots I_r^{n_r})&=&-((I_1^{n_1}\cdots I_r^{n_r}\mathcal O_Y)^d)\\
&=&
\sum_{\stackrel{d_1,\ldots, d_r\in \ZZ_{\ge 0},}{d_1+\cdots+d_r=d}}\frac{-d!}{d_1!\cdots d_r!}
((I_1\mathcal O_Y)^{d_1}\cdot\ldots\cdot (I_r\mathcal O_Y)^{d_r})
n_1^{d_1}\cdots n_r^{d_r}
\end{array}
$$
\end{Theorem}

The conclusions of Theorem \ref{MixMultIdeal}  are stated in \cite[Section 1.6B]{Laz1} under the assumption that  $R$ is the local ring of a closed point on a finite type scheme over an algebraically closed field. Related formulas  appear in \cite{KT}.

We  show that mixed multiplicities  exist for graded $m_R$-families in arbitrary Noetherian local rings $R$ in Section \ref{SecMult}.
 We prove the following theorem, generalizing Theorem \ref{MixMultIdeal} to graded $m_R$-families.
 
 \begin{Theorem}\label{Minkcoef*} Let $R$ be a $d$-dimensional Noetherian local ring and $\mathcal I(1),\ldots,\mathcal I(r)$ be graded $m_R$-families. Then there exists a homogeneous real polynomial $P(n_1,\ldots,n_r)$ of degree $d$ such that 
$$
\lim_{m\rightarrow\infty} \frac{e(I(1)_{mn_1}\cdots\ I(r)_{mn_r})}{m^d}=P(n_1,\ldots,n_r)
$$
for all $n_1,\ldots,n_r\in \ZZ_{\ge 0}$.
\end{Theorem}

Write
$$
P(n_1,\ldots,n_r)=\sum_{\stackrel{d_1,\ldots, d_r\in \ZZ_{\ge 0},}{d_1+\cdots+d_r=d}}    \frac{d!}{d_1!\cdots d_r!}e(\mathcal I(1)^{[d_1]},\ldots,\mathcal I(r)^{[d_r]})n_1^{d_1}\cdots n_r^{d_r}.
$$
The coefficients $e(\mathcal I(1)^{[d_1]},\ldots,\mathcal I(r)^{[d_r]})$ are called the mixed multiplicities of $\mathcal I(1),\ldots,\mathcal I(r)$.

Theorem \ref{Minkcoef*} is proven in \cite{CSS} with the additional assumptions that $\dim N(\hat R)<d$ and that the graded $m_R$-families are filtrations. The proof uses the theory of Okounkov bodies and equation (\ref{eq10}). The theorem is proven for graded $m_R$-families, also with the assumption that $\dim N(\hat R)<d$ and using Okounkov bodies, in \cite{CRM}.

 To establish Theorem \ref{Minkcoef*} in an elementary way, we need the full generality of Theorem \ref{Theorem70}, which says that if 
 $\mathcal I(1)=\{I(1)_n\}, \,\,\mathcal I(2)=\{I(2)_n\},\,\,\ldots,\mathcal I(r)=\{I(r)_n\}$ are graded $m_R$-families, then 
 with $I(i)_n\mathcal O_{Y(i)_n}=\mathcal O_{Y_n}(-F(i)_n)$ where $Y(i)_n$ is the blowup of $I(i)_n$,  the limit
\begin{equation}\label{eqIn1}
\lim_{n_1,\ldots,n_r\rightarrow\infty} \left((\frac{-F(1)_{n_1}}{n_1})^{d_1}\cdot\ldots\cdot (\frac{-F(r)_{n_r}}{n_r})^{d_r}\right)
\end{equation}
exists.
The proof of Theorem \ref{Minkcoef*} now follows from equation (\ref{eqIn1}) and the multilinearity of intersection products.
It is shown in the proof of Theorem \ref{Minkcoef*}, given after Lemma \ref{Prop71}, that
 \begin{equation}\label{MixInt}
\begin{array}{l}
\displaystyle{\lim_{m\rightarrow\infty}} \frac{e(I(1)_{mn_1}\cdots\ I(r)_{mn_r})}{m^d} =
\lim_{m\rightarrow\infty}-\left(\left(\frac{-F(1)_{mn_1}}{m}+\frac{-F(2)_{mn_2}}{m}+\cdots+\frac{-F(r)_{mn_r}}{m}\right)^d\right)\\
= \displaystyle \sum_{d_1+\ldots+d_r=d}\frac{-d!}{d_1!\cdots d_r!} \left(\lim_{m\rightarrow\infty}\left(\left(\frac{-F(1)_{m}}{m}\right)^{d_1}\cdot\ldots\cdot
\left(\frac{-F(r)_{m}}{m}\right)^d\right)\right)n_1^{d_1}\cdots n_r^{d_r}.
\end{array}
\end{equation}
 In particular, using Theorem \ref{MixMultIdeal} for the second equality, we have that 
\begin{equation}\label{eqIn15}
e(\mathcal I(1)^{[d_1]},\ldots,\mathcal I(r)^{[d_r]})=\lim_{n\rightarrow \infty}-((\frac{-F(1)_n}{n})^{d_1}\cdot\ldots\cdot (\frac{-F(r)_n}{n})^{d_r})
=\lim_{n\rightarrow\infty}\frac{e(I(1)_n^{[d_1]},\ldots,I(r)_n^{[d_r]})}{n^d}.
 \end{equation}
 As a consequence of Theorem \ref{Minkcoef*}, we show in Theorem \ref{MinkcoefM} that for a finitely generated $R$-module $M$,
$$
\lim_{m\rightarrow\infty} \frac{e(I(1)_{mn_1}\cdot\ldots\cdot\ I(r)_{mn_r};M)}{m^d}
 $$
 is a polynomial of degree $d$ in $n_1,\ldots,n_r$ for all $n_1,\ldots,n_r\in \ZZ_{\ge 0}$.

Let $\mathcal I=\{I_n\}$ be a graded $m_R$-family. 

\begin{equation}\label{eqIn3}
\mbox{Given $c\in \RR_{>0}$, let $\mathcal I^{(c)}$ be the family $\{I_{\lceil cn\rceil}\}$,}
\end{equation}
where $\lceil cn\rceil$ is the round up of the real number $cn$. If $c=0$, then $\mathcal I^{(0)}$ is the filtration such that $(I^{(0)})_n=m_R$ for $n>0$.

If $\mathcal I(1),\ldots,\mathcal I(r)$ are graded families, then
we define the product 
\begin{equation}\label{eqIn4}
\mathcal I(1)\mathcal I(2)\cdots\mathcal I(r)=\{I(1)_nI(2)_n\cdots I(r)_n\}.
\end{equation}

Using this language and equation (\ref{eqIn15}),
the mixed multiplicities of $\mathcal I(1),\ldots,\mathcal I(r)$ are
\begin{equation}\label{eqIn16}
\begin{array}{lll}
e(\mathcal I(1)^{[d_1]},\ldots,\mathcal I(r)^{[d_r]})&=&\lim_{n\rightarrow \infty}-((\frac{-F(1)_n}{n})^{d_1}\cdot\ldots\cdot (\frac{-F(r)_n}{n})^{d_r})\\
&=&-(\mathcal I(1)^{d_1}\cdot\ldots\cdot\mathcal I(r)^{d_r}),
\end{array}
\end{equation}
and we can restate the formula of Theorem \ref{Minkcoef*} as
\begin{equation}\label{eqIn17}
e(\mathcal I(1)^{(n_1)}\mathcal I(2)^{(n_2)}\cdots\mathcal I(r)^{(n_r)})=\sum_{d_1+\cdots+d_r=d} \frac{d!}{d_1!\cdots d_r!}
e(\mathcal I(1)^{d_1},\ldots,\mathcal I(r)^{d_r})n_1^{d_1}\cdots n_r^{d_r}.
\end{equation}

\subsection{Valuations, integral closure and saturation}\label{SubSecVal}

Let $R$ be a $d$-dimensional Noetherian local ring. A valuation $v$ of $R$ is defined  in Definition VI.3.1, page 386 \cite{Bou}
Bourbaki as follows. There is a prime ideal $P$ of $R$ and a valuation $v$ of the quotient field of $R/P$ such that $v$ is nonnegative on $R/P$. Elements of $P$ have value $\infty$, which is larger than any element of the value group of $v$.
A valuation $v$ of $R$ is an $m_R$-valuation (as  defined by Rees) if $v(m_R)>0$ and $v$ is a divisorial valuation of $R$; that is, the residue field of the valuation has maximal transcendence degree $d-1$ over the residue field $R/m_R$ of $R$ (Definition 9.3.1 \cite{HS}).
The value group of a divisorial valuation $v$ is $\ZZ$ (by Theorem 9.3.2 \cite{HS} and Proposition 6.4.4 \cite{HS}), so the valuation ring $\mathcal O_v$  uniquely determines the valuation $v$.
Let $V(R)$ be the set of $m_R$-valuations of $R$.  
There is a natural 1-1 correspondence between $V(R)$ and the disjoint union of $V(\hat R/P_i)$ where $P_i$ are the minimal prime ideals of $\hat R$ such that $\dim \hat R/P_i=d$ (a proof of this is given in Lemma \ref{LemVal}).

If $R$ is an analytically unramified local domain, such as a complete local domain, then $V(R)$ is equal to the set of  valuations dominating $R$ such that the valuation ring $\mathcal O_v$ is essentially of finite type over $R$ (this follows from Lemma 4.2 \cite{R} or Theorem 9.3.2 \cite{HS}). Given $v\in V(R)$, and assuming that $R$ is an analytically unramified local domain, there exists an $m_R$-primary ideal $I$ in $R$ and an integral exceptional   divisor $E$ on the blow up of $I$ 
such that the valuation ring of $v$ is $\mathcal O_v=\mathcal O_{Y,E}$, the local ring of  $E$ in $Y$, by the analysis in \cite{R2} or Proposition 4.9 \cite{CM}.  We will write $v_E$ for the  valuation of $\mathcal O_{Y,E}$ ($v_E(f)=n$ if $f$ is in the quotient field of $R$ and $f=t^nu$ where $t$ is a generator of the maximal ideal of $\mathcal O_{Y,E}$ and $u\in \mathcal O_{Y,E}$ is a unit).

 If $R$ is a Noetherian local ring and $v\in V(R)$, then we define a graded filtration by $\mathcal I(v)=\{I(v)_n\}$ where 
 \begin{equation}\label{eqvalfin}
 I(v)_n=\{ f\in R\mid v(f)\ge n\}.
 \end{equation}
  A (real) divisorial filtration is a graded filtration $\mathcal I=\{I_n\}$ where there exist $v_1,\ldots,v_s\in V(R)$ and $a_1,\ldots,a_s\in \RR_{\ge 0}$ such that $I_n=\{f\in R\mid v_i(f)\ge na_i\mbox{ for }1\le i\le s\}$.

Define $v(\mathcal I)$ for $\mathcal I$ a graded filtration of $m_R$-primary ideals and $v\in V(R)$ by
\begin{equation}\label{ex2*}
v(\mathcal I)=\inf_{m\in \ZZ_{>0}}\frac{v(I_m)}{m}\in \RR_{\ge 0}.
\end{equation}
 It is   shown in \cite{JM} that 
$v(\mathcal I)=\displaystyle\lim_{m\rightarrow\infty}\frac{v(I_m)}{m}$.

If $I$ is an ideal in a local ring $R$, then $\overline I$ denotes its integral closure. 

Suppose that $R$ is a local ring and $\mathcal I=\{I_n\}$ is a graded  $m_R$-family. The proof in \cite[Lemma 5.6]{C8} extends immediately 
to show that the integral closure $\overline{\sum_{n\ge 0}I_nt^n}$ of $\sum_{n\ge 0}I_nt^n$ in $R[t]$ is the graded $R$-algebra 
$\overline{\sum_{n\ge 0}I_nt^n}=\sum_{n\ge 0}I^*_nt^n$,  where $\overline{\mathcal I}=\{I^*_n\}$ is the filtration
$$
I^*_n=\{f\in R\mid f^r\in \overline{I_{rn}}\mbox{ for some }r>0\}.
$$
If $\mathcal I=\{I^n\}$ is the adic-filtration of the powers of a fixed ideal $I$, then $I^*_n=\overline{I^n}$ for all $n$. However, it is possible for the integral 
closure $\overline{\mathcal I}$ of a graded $m_R$-family $\mathcal I=\{I_n\}$ to strictly contain the graded $m_R$-family $\{\overline{I_n}\}$ of integral closures of the ideals $I_n$.
A simple example where this occurs is in the example given after Definition \ref{SatDef}.


The following definition is made in Definition 3.1 \cite{BLQ} on  analytically irreducible local rings $R$, so that $V(R)$ is the set of divisorial valuations of the quotient field of $R$ which dominate $R$. We extend the definition to Noetherian  local rings.
\begin{Definition}\label{SatDef} Let $R$ be a Noetherian local ring  and $\mathcal I$ be a graded $m_R$-family. The saturation $\tilde{\mathcal I}$ of $\mathcal I$ is the filtration
$\tilde{\mathcal I}=\{\tilde I_n\}$ defined by 
$$
\tilde I_n=\{f\in m_R\mid v(f)\ge n v(\mathcal I)\mbox{ for all }v\in V(R)\}.
$$
\end{Definition}

It is shown in Lemma \ref{SatInt} that 
the graded $R$-algebra $R[\tilde{\mathcal I}]=\sum_{n\ge 0}\tilde I_nt^n$ is integrally closed in $R[t]$.
  We have inclusions 
$$
\mathcal I\subset \{\overline{I_n}\}\subset\overline{\mathcal I}\subset \tilde{\mathcal I}
$$
where it may be that all of the inclusions are proper. A simple example where the last two   inclusions are proper is the following. Let $R=k[[t]]$, where $k$ is a field, be a power series ring. Let the graded filtration $\mathcal I=\{I_n\}$ be defined by
$$
I_n=\left\{\begin{array}{ll} (t^{n+2})&\mbox{ if $n$ is odd}\\
                                          (t^{n+1})&\mbox{ if $n$ is positive and even.}
                                          \end{array}\right.
$$
 Then $\overline{\mathcal I}=\{I^*_n\}$ is the filtration 
 $I^*_n=(t^{n+1})$ for $n>0$. The only valuation in $V(R)$  is the $m_R$-adic valuation $v$, defined by $v(f)=\mbox{order}(f)$ for $f\in R$. We compute $v(\mathcal I)=1$. Thus   $\tilde{\mathcal I}$ is the $m_R$-adic filtration, $\tilde I_n=(t^n)$ for all $n\in \ZZ_{>0}$.    $I_n$ is integrally closed for all $n$, so $\mathcal I=\{\overline{I_n}\}$ and thus 
  $\{\overline{I_n}\}\ne \overline{\mathcal I}\ne \tilde{\mathcal I}$.

Let $I$ be an $m_R$-primary ideal and $\mathcal I=\{I^n\}$ be the $I$-adic filtration. Rees showed that 
there are  discrete valuations $v_1,\ldots,v_r$ dominating $R$  and $a_1,\ldots,a_r\in\ZZ_{>0}$ such that 
$$
I^*_n=\overline{I_n}=\{f\in R\mid v_i(f)\ge na_i\mbox{ for }1\le i\le r\}
$$
for all $n\ge 0$. The valuations $v_i$ are  called the Rees valuations of $I$. If $R$ is formally equidimensional, then the Rees valuations are all elements of $V(R)$ (as follows from Proposition 10.2.6 \cite{HS})
and $\tilde{\mathcal I}=\overline{\mathcal I}$ is thus a divisorial filtration.

If $R$ is not formally equidimensional, then there exist $m_R$-primary ideals $I$ of $R$ such that the filtration
 $\{\overline{I^n}\}$ is not real divisorial, which will occur if the Rees valuations of $I$ are not all $m_R$-valuations.

It is more difficult to find examples where $\tilde{\mathcal I}$ is not real divisorial on a formally equidimensional local ring. 
We give such an example in Section \ref{Notdivex},  showing that there exists a saturated filtration on a two dimensional regular local ring which is not real divisorial.

\subsection{Rees's Theorem}\label{ReesSubSec}

In this subsection we compare filtrations $\mathcal I\subset \mathcal J$ for which $e(\mathcal I)=e(\mathcal J)$.
For $m_R$-primary ideals in a formally equidimensional local ring, there is a very nice characterization of when this condition holds, discovered and proven by Rees.

\begin{Theorem}(Rees \cite{R}, Theorem 11.3.1 \cite{HS})\label{ReesThm} Suppose that $R$ is a formally equidimensional $d$-dimensional Noetherian local ring and $I\subset J$ are $m_R$-primary ideals. Then $J\subset \overline{I}$ if and only if $e(J)=e(I)$.
\end{Theorem}

Since the integral closure of $R[It]$ in $R[t]$ is $\oplus_{n\ge 0}\overline {I^n}$, $J\subset \overline I$ if and only 
$\overline{R[It]} = \overline{R[Jt]}$.

The condition that $R$ is formally equidimensional is necessary in Rees theorem. 
If $R$ is a Noetherian local ring which is not formally equidimensional, then there exist $m_R$-primary ideals $I\subset J$ such that $e(I)=e(J)$ but $R[Jt]$ is not integral over $R[It]$.  Thus $J\not \subset \overline I$.   This follows from Exercise 11.5 and Theorem 8.2.1 \cite{HS}.

We  consider the generalization of Theorem \ref{ReesThm} to graded $m_R$-families.  We say that $\mathcal I=\{I_n\}\subset \mathcal J=\{J_n\}$ if $I_n\subset J_n$ for all $n$. We observe that if $\mathcal I\subset \mathcal J$, then
\begin{equation}\label{eq80}
e(\mathcal J)\le e(\mathcal I).
\end{equation}
This follows since $I_n\subset J_n$ implies $e(J_n)\le e(I_n)$. If $\mathcal I\subset \mathcal J$ then $\overline{\mathcal I}\subset \overline{\mathcal J}$ and $\tilde{\mathcal I}\subset \tilde{\mathcal J}$.

Example \ref{ReesEx} shows that it is possible that  $\tilde{\mathcal I}=\tilde{\mathcal J}$ but $\overline{\mathcal I}$ is not equal to $\overline{\mathcal J}$, even on an analytically irreducible local ring.

The following theorem in \cite{CSS} generalizes  \cite[Proposition 11.2.1]{HS} for $m_R$-primary ideals to filtrations of $R$ by $m_R$-primary ideals. 

\begin{Theorem}\label{Theorem13}(\cite[Theorem 6.9]{CSS}, Theorem 1.4 \cite{C9}) Suppose that $R$ is a Noetherian $d$-dimensional local ring such that 
$
\dim N(\hat R)<d.
$
 Suppose that $\mathcal I\subset \mathcal J$ are graded $m_R$-filtrations  and the ring $\bigoplus_{n\ge 0}I_n$ is integral over $\bigoplus_{n\ge 0} J_n$. Then 
$$
e(\mathcal I)=e(\mathcal J).
$$
\end{Theorem}

However, the converse of Theorem \ref{Theorem13} does not hold. 
The following example shows that even on an analytically irreducible local ring, we can find $\mathcal I\subset \mathcal J$ such that $e(\mathcal I)=e(\mathcal J)$ but 
$R[\overline{\mathcal I}]\ne R[\overline{\mathcal J}]$.

\begin{Example}\label{ReesEx} In the power series ring $R=k[[x]]$ over a field $k$, let $\mathcal I=\{I_n\}$ 
where $I_n=(x^{n+1})$ and $\mathcal J=\{J_n\}$ where $J_n=(x^n)$.  We have that $\overline{\mathcal I}=\mathcal I$ and $\overline{\mathcal J}=\mathcal J$. Thus 
$\overline{\mathcal I}$ is not equal to $\overline{\mathcal J}$ although $\mathcal I$ and $\mathcal J$ have the same multiplicity
 $e(\mathcal I)=e(\mathcal J)=1$.   In the example, $V(R)$ consists of only the $m_R$-adic valuation $v$ and  $v(\mathcal I)=v(\mathcal J)$  so $\tilde{\mathcal I}=\tilde{\mathcal J}$.

\end{Example}

There is one situation where the converse to Theorem \ref{Theorem13} holds. A graded $m_R$-filtration $\mathcal I$ is bounded if its integral closure $\overline{\mathcal I}$ is a divisorial filtration. The $I$-adic filtration of an $m_R$-primary ideal $I$ on a formally equidimensional local ring is an example of a  bounded filtration.

\begin{Theorem}\label{NewRBT}(Theorem 1.2 \cite{C8}) Suppose that $R$ is an excellent local domain or an analytically irreducible local domain, $\mathcal I$ is a   bounded $m_R$-filtration and $\mathcal J$ is an arbitrary $m_R$-filtration such that $\mathcal I\subset \mathcal J$. Then the following are equivalent
\begin{enumerate}
\item[1)]  $e(\mathcal I)=e(\mathcal J)$.
\item[2)]  $\overline{\mathcal I}=\overline{\mathcal J}$; that is, 
there is equality of  integral closures  
$
R[\overline{\mathcal I}]=R[\overline{\mathcal J}].
$
\end{enumerate}
 \end{Theorem}

In  \cite{BLQ}, it is shown that for graded $m_R$-filtrations $\mathcal I\subset \mathcal J$ on an analytically irreducible local ring,
$e(\mathcal I)=e(\mathcal J)$ if and only if $\tilde{\mathcal I}=\tilde{\mathcal J}$,
generalizing  Rees's Theorem in this situation.  We generalize their result to graded $m_R$-families on an arbitrary Noetherian local ring to give the following general theorem.

\begin{Theorem}\label{Theorem14*} (proven for the essential case of analytically irreducible local rings in \cite{BLQ}) Suppose that $R$ is a Noetherian local ring and $\mathcal I\subset \mathcal J$ are graded $m_R$-families. Then $e(\mathcal I)=e(\mathcal J)$ if and only if $\tilde{\mathcal I}=\tilde{\mathcal J}$. In particular, $\tilde{\mathcal I}$ is the largest $m_R$-filtration $\mathcal J$ such that $\mathcal I\subset \mathcal J$ and $e(\mathcal J)=e(\mathcal I)$.
\end{Theorem}

Theorems \ref{ReesThm}, \ref{Theorem13} and \ref{NewRBT} are deduced as corollaries of Theorem \ref{Theorem14*} in Corollary \ref{ReesThmpf}, Corollary \ref{ReesInt} and  \ref{ReesDiv} respectively.

In \cite{BLQ}, they prove Theorem \ref{Theorem14*} for  volume $\mbox{vol}(\mathcal I)$ instead of for our definition of multiplicity $e(\mathcal I)$. 
 Since they are in an analytically irreducible local ring,  the volume = multiplicity formula
$$
\mbox{\rm vol}(\mathcal I)=\lim_{p\rightarrow\infty}\frac{e(I_p)}{p^d}
$$
of Theorem \ref{volmult} holds, so their theorem implies Theorem \ref{Theorem14*} for analytically irreducible local rings. 

We modify the proof in \cite{BLQ} for analytically irreducible local rings to give a proof of Theorem \ref{Theorem14*} that is independent of the theory of volume and of Okounkov bodies. We accomplish this in Section \ref{ReesSec}.  There are three results going into the proof of Theorem \ref{Theorem14*}.  
The proofs of Lemma \ref{Lemma21} and Theorem \ref{Theorem312} require some alteration from the proofs in \cite{BLQ} to reflect the difference between the definitions of volume and multiplicity. 
 Their proof of Theorem \ref{Theorem313}  uses the volume = multiplicity formula, Theorem \ref{volmult}, to work with our definition of multiplicity $e(\mathcal I)$ instead of using $\mbox{vol}(\mathcal I)$. As such it is independent of results on volume and Okounkov bodies.  The essential difficulty in the proof is to get around the fact that given a finite set of valuations $w_1,\ldots, w_s$ in $V(R)$, we can find an $n$ such that the $\frac{w_i(I_n)}{n}$ are very close to $w_i(\mathcal I)$, but we may not be able to find an $n$ such that $v_E(I_n)$ is close to $v_E(\mathcal I)$ for all $E$ which are exceptional divisors on the blowup of $I_n$.

\subsection{Minkowski inequalities and equality}\label{SubsecMink}


We extend the Minkowski inequalities to  graded $m_R$-families on arbitrary Noetherian local rings. The proof of the following theorem is given in Section \ref{SecMink}. 

\begin{Theorem}\label{MinkIneq*}(Minkowski Inequalities)  Suppose that $R$ is a Noetherian $d$-dimensional  local ring and  $\mathcal I=\{I_n\}$ and $\mathcal J=\{J_n\}$ are graded $m_R$-families. Let $e_i=e(\mathcal I^{[d-i]},\mathcal J^{[i]})$ for $0\le i\le d$.
Then 
\begin{enumerate}
\item[1)] $e_i^2\le e_{i-1}e_{i+1}$  for $1\le i\le d-1$,
\item[2)] $e_ie_{d-i}\le e_0e_d$ for $0\le i\le d$,
\item[3)]  $e_i^d\le d_0^{d-i}e_d^i$ for $0\le i\le d$,
\item[4)] $e(\mathcal I\mathcal J)^{\frac{1}{d}}\le e_0^{\frac{1}{d}}+e_d^{\frac{1}{d}}$
where $\mathcal I\mathcal J=\{I_nJ_n\}$.
\end{enumerate}
\end{Theorem}

The Minkowski inequalities were formulated and proven for  $m_R$-primary ideals $I$ and $J$ (so that $\mathcal I=\{I^n\}$ and $\mathcal J=\{J^n\}$ are adic filtrations) in a Noetherian local ring $R$ by Teissier \cite{T1}, \cite{T2} and  Rees and Sharp \cite{RS}.

The fourth inequality  4)  was proven for filtrations  of $R$ by $m_R$-primary ideals in a regular local ring with algebraically closed residue field by Musta\c{t}\u{a} (\cite[Corollary 1.9]{Mus}) and more recently by Kaveh and Khovanskii (\cite[Corollary 7.14]{KK1}) with the same assumptions.
Theorem \ref{MinkIneq*} is proven for graded $m_R$-filtrations on a local ring $R$ with $\dim N(\hat R)<d$ in Theorem 6.3 \cite{CSS} and in Theorem 4.3 \cite{CRM} for graded $m_R$-families on a local ring $R$ with $\dim N(\hat R)<d$.

Teissier \cite{T3} (for Cohen-Macaulay normal two-dimensional complex analytic $R$), Rees and Sharp \cite{RS} (in dimension 2) and Katz \cite{Ka} (in complete generality) have proven that if $R$ is a $d$-dimensional formally equidimensional Noetherian local ring with $d\ge 2$ and $I$ and $J$ are $m_R$-primary ideals such that the Minkowski equality
$$
e((IJ))^{\frac{1}{d}}= e( I)^{\frac{1}{d}}+e(J)^{\frac{1}{d}}
$$
holds,
 then there exist positive integers $r$ and $s$ such that the integral closures 
 $\overline{I^r}$ and $\overline{J^s}$ of the ideals $I^r$ and $J^s$ are equal, which is equivalent to the statement that the $R$-algebras $\bigoplus_{n\ge 0}I^{rn}$ and $\bigoplus_{n\ge 0}J^{sn}$ have the same integral closure.

 The Teissier-Rees-Sharp-Katz theorem is not true for graded $m_R$-families, even in a regular local ring, as is shown in a simple example in  Remark 6.4 \cite{CSS}, which we reproduce here. Let $k$ be a field and $R=k[[x_1,\ldots,x_n]]$ be the power series ring. Let $\mathcal I=\{I_n\}$ where $I_n=m_R^{n+1}$ and $\mathcal J=\{J_n\}$ where $J_n=m_R^{n}$. Then 
 $$
 e(\mathcal I\mathcal J)^{\frac{1}{d}}=e(\mathcal I)^{\frac{1}{d}}+e(\mathcal J)^{\frac{1}{d}} 
 $$
 but $\overline{\mathcal I}\ne \overline{\mathcal J}$. 
 
 However, this theorem is true for $\QQ$-bounded $m_R$-filtrations $(\overline{\mathcal I}$ is a $\QQ$-divisorial $m_R$-filtration) as shown in the following theorem.
 
 \begin{Theorem}\label{BoundMink}(Proven in Theorem 1.4 \cite{C8} when $R$ is analytically irreducible  or an excellent local domain)
  Suppose that $R$ is a $d$-dimensional Noetherian  local ring of dimension $d\ge 2$ and $\mathcal I$ and $\mathcal J$ are $\QQ$-bounded $m_R$-filtrations.
Then the  Minkowski equality
$$
e(\mathcal I\mathcal J)^{\frac{1}{d}}=e(\mathcal I)^{\frac{1}{d}}+e(\mathcal J)^{\frac{1}{d}}
$$
holds if and only if there exist positive integers a,b such that there is equality of integral closures
$\overline{\mathcal I^{(a)}}=\overline{\mathcal J^{(b)}}$.
\end{Theorem}

A proof of Theorem \ref{BoundMink} will be given in Section \ref{SecMinkbound}, using Theorem \ref{BrunnMink}.

 The following theorem is proven in 
Theorem 10.3 \cite{C8}, using the theory of Okounkov bodies and the Brunn-Minkowski theorem of convex geometry. Although Theorem 10.3 \cite{C8} assumes that $\mathcal I$ is a filtration, the proof is valid for arbitrary graded $m_R$-families.  

 \begin{Theorem}\label{BrunnMink}(Theorem 10.3 \cite{C8})
  Suppose that $R$ is a $d$-dimensional analytically irreducible local ring or an excellent local domain of dimension $d\ge 2$
  and $\mathcal I$, $\mathcal J$ are graded $m_R$-families such that $e(\mathcal I)>0$, $e(\mathcal J)>0$  and equality holds in the Minkowski inequality
$$
e(\mathcal I\mathcal J)^{\frac{1}{d}}=e(\mathcal I)^{\frac{1}{d}}+e(\mathcal J)^{\frac{1}{d}}.
$$
Then 
$$
e(\mathcal J)^{\frac{1}{d}}v(\mathcal I)=e(\mathcal I)^{\frac{1}{d}}v(\mathcal J)
$$
for all $v\in V(R)$.
\end{Theorem}

 We extend this  theorem to arbitrary Noetherian local rings  of dimension $d\ge 2$ in Theorem \ref{Mink*}.  The proof is by reduction to the case that $R$ is analytically irreducible 
 with $e(\mathcal I)>0$ and $e(\mathcal J)>0$, so that Theorem \ref{BrunnMink} can be applied. To make this reduction we must assume that $R$ has dimension $\ge 2$.  The analysis of the Minkowski inequalities in Section 9 of \cite{C8} require that there are at least 3 mixed multiplicities, so that $\dim R\ge 2$.   The proof of Theorem \ref{Mink*} is given in Section \ref{SecMink}. Theorem \ref{Mink*} is not valid for one dimensional local rings which are not analytically irreducible. An example showing this is given in Example \ref{ExOned}.

 \begin{Theorem}\label{Mink*} Suppose that $R$ is a $d$-dimensional Noetherian local ring with $d\ge 2$ and $\mathcal I$, $\mathcal J$ are graded $m_R$-families. Then   equality holds in the Minkowski inequality
$$
e(\mathcal I\mathcal J)^{\frac{1}{d}}=e(\mathcal I)^{\frac{1}{d}}+e(\mathcal J)^{\frac{1}{d}}
$$
if and only if 
$$
e(\mathcal J)^{\frac{1}{d}}v(\mathcal I)=e(\mathcal I)^{\frac{1}{d}}v(\mathcal J)
$$
for all $v\in V(R)$.
\end{Theorem}

 We have the following general result characterizing equality in the Minkowski inequality, which will be proven in Corollary \ref{Minkgen}.
Corollary \ref{Minkgen*} is not valid for one dimensional local rings which are not analytically irreducible. Example \ref{ExOned} of Section \ref{SecMink} shows this.

\begin{Theorem}\label{Minkgen*} Suppose that $R$ is a $d$-dimensional Noetherian local ring with $d\ge 2$ and $\mathcal I$, $\mathcal J$ are graded $m_R$-families such that 
$e(\mathcal J)>0$. 
Then there is equality in the Minkowski inequality
$$
e(\mathcal I\mathcal J)^{\frac{1}{d}}= e(\mathcal I)^{\frac{1}{d}}+e(\mathcal J)^{\frac{1}{d}}
$$
if and only if $\tilde{\mathcal I}=\widetilde{(\mathcal J^{(c)})}$ for some $c\in \RR_{\ge 0}$.
\end{Theorem}

Corollary \ref{Minkgen*} is proven in Corollary 1.5  \cite{BLQ} when $R$ is analytically irreducible, using Theorem 10.3 \cite{C8} (Theorem \ref{BrunnMink} of this paper). 



\subsection{One dimensional Noetherian local rings}

Let $R$ be a one dimensional Noetherian local ring. Suppose that $I$ and $J$ are $m_R$-primary ideals. Then
$$
e(IJ)=e(I)+e(J)
$$ 
by Proposition 4.3 \cite{GGV}. Thus if $\mathcal I$ and $\mathcal J$ are graded $m_R$-families, then 
\begin{equation}\label{onedim}
e(\mathcal I\mathcal J)=e(\mathcal I)+e(\mathcal J),
\end{equation}
that is, the Minkowski equality is satisfied for all graded $m_R$-families. 

 If $R$ is a complete local domain, then $R$ is excellent and Henselian, so its normalization $A$ is a finite $R$-module (Theorem 9.9.2 \cite{HS}) and is a one dimensional regular local ring (Theorem 4.9 \cite{Mi}). Thus $A$ is a discrete valuation ring and  $V(R)=\{v\}$ where $v$ is the valuation of $A$. 

Suppose that $R$ is a one dimensional local ring with no other assumptions. Let $P_1,\ldots,P_r$ be the minimal primes of $\hat R$. Then $V(R)=\{v_1,\ldots,v_r\}$ where $V(\hat R/P_i)=\{v_i\}$ by Lemma \ref{Lemma4}. Thus if $\mathcal I=\{I_n\}$ is a graded $m_R$-family, then 
$$
\tilde I_n=\{f\in R\mid v_i(f)\ge nv_i(\mathcal I)\mbox{ for }1\le i\le r\}.
$$

If $R$ is analytically irreducible, so that $V(R)=\{v\}$ where $\mathcal O_v$ is the normalization of $R$, then $\tilde{\mathcal I}=\mathcal I(v)^{(v(\mathcal I))}$. In particular, the conclusions of Theorem \ref{Minkgen*} hold if $R$ is one dimensional and analytically irreducible. 

\begin{Example}\label{ExOned} If $R$ is a one dimensional Noetherian local ring which is not analytically irreducible, then the conclusions of Theorems \ref{Mink*} and \ref{Minkgen*} may not hold, even for adic filtrations.
\end{Example}

\begin{proof} Let $k$ be a field of characteristic not equal to 2 and $R=k[[x,y]]/(y^2-x^2)$, a quotient of a power series ring. Let $n$ be an integer larger than one and let 
$I=(y-x,x^n)$ and $J=(y+x),x^n)$. Let $\mathcal I=\{I^m\}$ and $\mathcal J=\{J^m\}$.

$V(R)=\{v,w\}$ where $v$ is the $x$-adic valuation of $R/(y-x)\cong k[[x]]$ and $w$ 
is the $x$-adic valuation of $R/(y+x)\cong k[[x]]$. $v(I)=v(IR/(y-x))=n$ and $v(J)=v(JR/(y-x))=1$. Thus $v(\mathcal I)=n$ and $v(\mathcal J)=1$. Similarly, $w(\mathcal I)=1$ and $w(\mathcal J)=n$. 
$e(I)=e(IR(y-x))+e(IR/(y+x))=n+1$ and similarly, $e(J)=1+n$ for $n\ge 1$. Thus
$\displaystyle e(\mathcal I)=n+1$ and $e(\mathcal J)=n+1$. Therefore
$$
e(\mathcal J)v(\mathcal I)=(n+1)n\ne n+1=e(\mathcal I)v(\mathcal J)\mbox{ and }
e(\mathcal J)w(\mathcal I)=n+1\ne (n+1)n=e(\mathcal I)v(\mathcal J),
$$
so the conclusions of Theorem \ref{Mink*} and Corollary \ref{Minkgen*} do not hold, even though the Minkowski equality $e(\mathcal I\mathcal J)=e(\mathcal I)+e(\mathcal J)$ holds, by (\ref{onedim}).
\end{proof}

\section{Intersection products of graded $m_R$-families}
In this section we assume that 
$R$ is a $d$-dimensional Noetherian local ring and $\mathcal I(j)=\{I(j)_n\}$ are graded $m_R$-families, for $1\le j\le r$.

Let $\mbox{BirMod}(R)$ be the directed set of blowups $\phi:X\rightarrow \mbox{Spec}(R)$ of  $m_R$-primary ideals of $R$.

Let $\phi:X\rightarrow\mbox{Spec}(R)\in \mbox{BirMod}(R)$. 
We refer to Section 9 of \cite{Mum} for   the theory of Cartier divisors on a scheme. We will say that a Cartier divisor $D$ on $X$ is effective if $\mathcal O_X(-D)\subset \mathcal O_X$ (page 63 \cite{Mum}), and 
that  a Cartier divisor on $X$ has exceptional support if
the support of $\mathcal O_X(D)$ is contained in $\phi^{-1}(m_R)$. The support of $\mathcal O_X(D)$ is the (closed) set of points $p\in X$ such that a local equation of $D$ at $p$ is not in $\mathcal O_{X,p}^*$ (B.4.2 \cite{Fl}).
Let $\mbox{Div}(X)$ be the set of Cartier divisors on $X$ with exceptional support. 

We will use the intersection theory on $X$ of \cite{CM}, giving  products
$(D_1\cdot\ldots \cdot D_{d})\in \ZZ$  and $(D_1\cdot\ldots\cdot D_{d-r}\cdot V)\in \ZZ$ for $D_1,\ldots,D_d\in \mbox{Div}(X)$, and $V$
an integral subscheme of $X$ of dimension $r$ which contracts to $m_R$.  The definition and important properties of this product are given an essentially self contained derivation in Chapter 2 of \cite{CM}.
To develop this product, which is valid over an arbitrary Noetherian local ring $R$, we use the theory of rational equivalence of Thorup for finite type schemes over a Noetherian base \cite{Th}. We also adapt some of the material from the intersection theory over
fields in Fulton’s book \cite{Fl} and from Kleiman's article \cite{Kl} which extends to our more general setting.

The intersection product extends multilinearly to $\mbox{Div}(X)_{\RR}:=\mbox{Div}(X)\otimes\mathcal \RR$.
We will say that $D\in \mbox{Div}(X)_{\RR}$ is effective if  
$D$ has an expression $D=\sum a_iE_i$ where $E_i\in \mbox{Div}(X)$ are effective Cartier divisors and all $a_i\in \RR_{>0}$.

Suppose that $X, Y\in \mbox{BirMod}(R)$, $\phi:Y\rightarrow X$ is an $R$-morphism  and $D_1,\ldots,D_d\in \mbox{Div}(X)_{\RR}$.  Then 
$\phi^*(D_i)\in \mbox{Div}(Y)_{\RR}$ for all $i$ and
\begin{equation}
(D_1\cdot\ldots\cdot D_d)=(\phi^*(D_1)\cdot\ldots\cdot \phi^*(D_d))
\end{equation}
by Proposition 2.8 \cite{CM}.
Thus for $Y_1,\ldots, Y_d\in \mbox{BirMod}(R)$ and Cartier divisors $D_i\in \mbox{Div}(Y_i)_{\RR}$ for $1\le i\le d$, 
 we have a well defined product $(D_1\cdot\ldots\cdot D_d)\in \RR$, defined as 
$$
(D_1\cdot\ldots\cdot D_d)=(\phi_1^*(D_1)\cdot\ldots\cdot \phi_d^*(D_d))
$$
If $Z\in \mbox{Birmod}(R)$ and $\phi_i:Z\rightarrow Y_i$ for $1\le i\le d$.

Suppose that $D_1\in \mbox{Div}(Y_1)_{\RR}$ and $D_2\in \mbox{Div}(Y_2)_{\RR}$. 
 We will say that $D_1\ge D_2$ if 
there exists $Z\in \mbox{BirMod}(R)$ and $R$-morphisms $\phi_1:Z\rightarrow Y_1$ and $\phi_2:Z\rightarrow Y_2$ such that 
$\phi_1^*(D_1)- \phi_2^*(D_2)$ is effective on $Z$.

A divisor $D\in \mbox{Div}(X)_{\RR}$ is nef (numerically effective) if $(D\cdot C)=0$ for all exceptional curves $C$ on $X$ (closed integral curves that contract to $m_R$). This concept is developed in Section 1.4 of \cite{Laz1} on algebraic schemes, and the proofs extend to our relative situation.
If $D\in \mbox{Div}(X)_{\RR}$ is nef and $Z\in \mbox{BirMod}(R)$ is such that there is an $R$-morphism $\phi:Z\rightarrow X$, then $\phi^*(D)\in \mbox{Div}(Z)_{\RR}$ is nef.

The proof of the following lemma is as in Example 1.4.16 \cite{Laz1}.

\begin{Lemma}\label{NefLem}
If $V$ is an $r$ dimensional integral subscheme  of $X\in \mbox{BirMod}(R)$ which contracts to $m_R$ and $G_1,\ldots, G_{d-r}\in \mbox{Nef}(X)$, then 
$(G_1\cdot\ldots\cdot G_{d-r}\cdot V)\ge 0$.
\end{Lemma}

\begin{Lemma}\label{Lemma4}
Suppose that $Y\in \mbox{BirMod}(R)$ and 
$F_i,G_i\in {\rm Div}(Y)_{\RR}$ for $1\le i\le d$ are nef and $F_i\ge G_i$ for $1\le i\le d$. Then
$$
(F_1\cdot\ldots\cdot F_d)\ge (G_1\cdot\ldots\cdot G_d).
$$
\end{Lemma}

\begin{proof} We have that 
$$
\left(F_1\cdot\ldots\cdot F_{i-1}\cdot(F_i-G_i)\cdot G_{i+1}\cdot\ldots\cdot G_d\right)\ge 0
$$
for $1\le i\le d$ by Lemma \ref{NefLem}.
since $F_i-G_i\ge 0$ and the $F_j$ and $G_j$ are nef.
The lemma now follows from  multilinearity of the intersection product.
\end{proof}

\begin{Theorem}\label{Theorem70} 
Suppose that $R$ is a $d$-dimensional Noetherian local ring with $d\ge 1$ and  $\mathcal I(1),\ldots,\mathcal I(r)$ are graded $m_R$-families. 
Let $Y(j)_n\in \mbox{BirMod}(R)$  be the blowup of $I(j)_n$. Define a Cartier divisor $F(j)_n$ on $Y(j)_n$ by $I(j)_n\mathcal O_{Y(j)_n}=\mathcal O_{Y(j)_n}(-F(j)_n)$.
$F(j)_n$ is effective and $-F(j)_n$ is ample on $Y(j)_n$, hence nef. 

Suppose that  $d_1,\ldots,d_r\in \ZZ_{\ge 0}$ are such that $d_1+\cdots+d_r=d$. Then
$$
\begin{array}{l}
\displaystyle \lim_{n_1,\ldots,n_r\rightarrow\infty} \left((\frac{-F(1)_{n_1}}{n_1})^{d_1}\cdot\ldots\cdot (\frac{-F(r)_{n_r}}{n_r})^{d_r}\right)\\
\,\,\,\,\,=\sup_{n_1,\ldots,n_r} \left\{\left((\frac{-F(1)_{n_1}}{n_1})^{d_1}\cdot\ldots\cdot (\frac{-F(r)_{n_r}}{n_r})^{d_r}\right)\right\}.
\end{array}
$$
In particular, the limit exists.
We will denote the above limit by
$(\mathcal I(1)^{d_1}\cdot\ldots\cdot \mathcal I(r)^{d_r})$. 
\end{Theorem}

\begin{proof}
Let
$$
S=\left\{\left((\frac{-F(1)_{n_1}}{n_1})^{d_1}\cdot\ldots\cdot (\frac{-F(r)_{n_r}}{n_r})^{d_r}\right)\mid n_1,\ldots,n_r\in \ZZ_{>0}\right\}.
$$
Each $F(i)_{n_i}$ is effective and the restriction of $\frac{-F(1)_{n_1}}{n_1},\ldots, \frac{-F(r)_{n_r}}{n_r}$ to $F(i)_{n_i}$ are nef, so 
$S\subset \RR_{\le0}$ by Lemma \ref{NefLem}. Let $c$ be the least upper bound of the set $S$. Necessarily, $c\le 0$.

Given $\epsilon>0$, there exists $m_1,\ldots,m_r\in \ZZ_{>0}$ such that
$$
 \left((\frac{-F(1)_{m_1}}{m_1})^{d_1}\cdot\ldots\cdot (\frac{-F(r)_{m_r}}{m_r})^{d_r}\right)> c-\frac{\epsilon}{2}.
 $$
 $I(j)_{m_j}^{n_j}\subset I(j)_{m_jn_j}$
 implies
 \begin{equation}\label{eq80'}
 -\frac{1}{m_jn_j}F(j)_{m_jn_j}\ge -\frac{1}{m_j}F(j)_{m_j}
 \end{equation}
 for all $n_j\in \ZZ_{>0}$, so  that for $n_1,\ldots,n_r\in \ZZ_{>0}$,
 \begin{equation}\label{eq71}
 \left((\frac{-F(1)_{m_1n_1}}{m_1n_1})^{d_1}\cdot\ldots\cdot (\frac{-F(r)_{m_rn_r}}{m_rn_r})^{d_r}\right)> c-\frac{\epsilon}{2}
 \end{equation}
  by Lemma \ref{Lemma4}.

Suppose that $m,i\in \ZZ_{>0}$. $I(j)_mI(j)_i\subset I(j)_{m+i}$ implies
\begin{equation}\label{eq5'}
-F(j)_i-F(j)_m\le -F(j)_{m+i}
\end{equation}
and
\begin{equation}\label{eq6'}
\begin{array}{lll}
\frac{-1}{m+i}F(j)_i-\frac{1}{m+i}F(j)_m&=&\frac{-1}{m+i}F(j)_i-(\frac{m}{m+i})(\frac{1}{m}F(j)_m)\\
&=& \frac{-1}{m+i}F(j)_i-(\frac{1}{m}F(j)_m)[\frac{m+i}{m+i}-\frac{i}{m+i}]\\
&=&[\frac{1}{m+i}(-F(j)_i+i(\frac{1}{m}F(j)_m)]-\frac{1}{m}F(j)_m.
\end{array}
\end{equation}
Thus $\frac{1}{m+i}(-F(j)_i+i(\frac{1}{m}F(j)_m))-\frac{1}{m}F(j)_m$ is nef and by (\ref{eq5'}),
\begin{equation}\label{eq7'}
\frac{1}{m+i}\left[-F(j)_i+i\left(\frac{1}{m}F(j)_m\right)\right]-\left(\frac{1}{m}F(j)_m\right)\le \frac{-1}{m+i}F(j)_{m+i}.
\end{equation}

For fixed $i_1,\ldots,i_r$ in $\ZZ$, with $0\le i_j<m_j$ for all $j$, 
there exists a  polynomial 
$$
Q_{i_1,\ldots,i_r}(x_1,y_1,z_1,\ldots,x_r,y_r,z_r)\in \ZZ[x_1,y_1,z_1,\ldots,x_r,y_r,z_r]
$$
  such that 
$$
Q_{i_1,\ldots,i_r}(x_1,y_1,z_1,\ldots,x_r,y_r,z_r)=\sum a_{\alpha_1,\beta_1,\gamma_1,\ldots,\alpha_r,\beta_r,\gamma_r}
x_1^{\alpha_1}y_1^{\beta_1}z_1^{\gamma_1}\cdots x_r^{\alpha_r}y_r^{\beta_r}z_r^{\gamma_r}
$$
with $\alpha_j+\beta_j=d_j$ for all $j$ and $\gamma_1+\cdots+\gamma_r>0$ if $a_{\alpha_1,\beta_1,\gamma_1,\ldots,\alpha_r,\beta_r,\gamma_r}\ne 0$, such that setting 
$$
\overline x_j=F(j)_{i_j}, \overline y_j=\frac{1}{m_jn_j}F(j)_{m_jn_j}, \overline z_j=\frac{1}{m_jn_j+i_j},
$$
we have that
\begin{equation}\label{eq70}
\begin{array}{lll}
&&Q_{i_1,\ldots,i_r}(\overline x_1,\overline y_1,\overline z_1,\ldots,\overline x_r,\overline y_r, \overline z_r)
+\left(\left(-\frac{1}{m_1n_1}F(1)_{m_1n_1}\right)^{d_1}\cdot\ldots\cdot \left(-\frac{1}{m_rn_r}F(r)_{m_rn_r}\right)^{d_r}\right)\\
&=& 
\left(\left(\frac{1}{m_1n_1+i_1}\left[-F(1)_{i_1}+\frac{i_1}{m_1n_1}F(1)_{m_1n_1}\right]-\frac{1}{m_1n_1}F(1)_{m_1n_1}\right)^{d_1}\right.\\
&&\cdot\ldots\cdot
\left.\left(\frac{1}{m_rn_r+i_r}\left[-F(r)_{i_r}+\frac{i_r}{m_rn_r}F(r)_{m_rn_r}\right]-\frac{1}{m_rn_r}F(r)_{m_rn_r}\right)^{d_r}\right)
\\
&\le& \left(\left(-\frac{1}{m_1n_1+i_1}F(1)_{m_1n_1+i_1}\right)^{d_1}\cdot\ldots\left(-\frac{1}{m_rn_r+i_r}F(r)_{m_rn_r+i_1}\right)^{d_r} \right)<c
\end{array}
\end{equation}
where the inequality is by equation (\ref{eq7'}) and Lemma \ref{Lemma4}.

If  $a_{\alpha_1,\beta_1,\gamma_1,\ldots,\alpha_r,\beta_r,\gamma_r}\ne 0$, then the absolute value
$$
\begin{array}{lll}
&&| \overline x_1^{\alpha_1}\overline  y_1^{\beta_1}\overline x_2^{\alpha_2}\overline  y_2^{\beta_2}
\cdots \overline x_r^{\alpha_r}\overline y_r^{\beta_r}|\\
&=&|\left((F(1)_{i_1})^{\alpha_1}\cdot(\frac{1}{m_1n_1}F(1)_{m_1n_1})^{\beta_1}\cdot\ldots\cdot
(F(r)_{i_r})^{\alpha_r}\cdot(\frac{1}{m_rn_r}F(r)_{m_rn_r})^{\beta_r}\right)|\\
&\le& |\left((F(1)_{i_1})^{\alpha_1}\cdot(\frac{1}{m_1}F(1)_{m_1})^{\beta_1}\cdot\ldots\cdot 
(F(r)_{i_r})^{\alpha_r}\cdot(\frac{1}{m_r}F(r)_{m_r})^{\beta_r}\right)|
\end{array}
$$
where the inequality is by (\ref{eq80'}) and Lemma \ref{Lemma4}.
Let 
$$
\lambda=\max\left\{\left|\left((F(1)_{i_1})^{\alpha_1}\cdot(\frac{1}{m_1}F(1)_{m_1})^{\beta_1}\cdot\ldots\cdot 
(F(r)_{i_r})^{\alpha_r}\cdot(\frac{1}{m_r}F(r)_{m_r})^{\beta_r}\right)\right|\right\}
$$ 
where the maximum is over 
$a_j+\beta_j=d_j$ for $1\le j\le r$.

Then whenever $a_{\alpha_1,\beta_1,\gamma_1,\ldots,\alpha_r,\beta_r,\gamma_r}\ne 0$,
$$
|\overline x_1^{\alpha_1}\overline y_1^{\beta_1}\overline z_1^{\gamma_1}\cdots 
\overline x_r^{\alpha_r}\overline y_r^{\beta_r}\overline z_r^{\gamma_r}|\le \lambda \prod_{j=1}^r(\frac{1}{m_jn_j+i_j})^{\gamma_j},
$$
with $\gamma_1+\cdots+\gamma_r>0$. Thus given $\delta>0$, there exists $n^*$ such that $n_j>n^*$ for all $j$ implies 
$$
|a_{\alpha_1,\beta_1,\gamma_1,\ldots,\alpha_r,\beta_r,\gamma_r}\overline x_1^{\alpha_1}\overline y_1^{\beta_1}\overline z_1^{\gamma_1}\cdots 
\overline x_r^{\alpha_r}\overline y_r^{\beta_r}\overline z_r^{\gamma_r}|<\delta
$$
for all $\alpha_1,\beta_1,\gamma_1,\ldots,\alpha_r,\beta_r,\gamma_r$. Choosing $\delta$ sufficiently small, we then ensure that
\begin{equation}\label{eq82'}
|Q_{i_1,\ldots,i_r}(\overline x_1,\overline y_1,\overline z_1,\ldots,\overline x_r,\overline y_r,\overline z_r)|< \frac{\epsilon}{2}
\end{equation}
if $n_j>n^*$ for all $j$. 

Let 
$$
\begin{array}{l}
Q=Q_{i_1,\ldots,i_r}(\overline x_1,\overline y_1,\overline z_1,\ldots,\overline x_r,\overline y_r,\overline z_r),\\
A=\left((-\frac{1}{m_1n_1}F(1)_{m_1n_1})^{d_1}\cdot\ldots\cdot (-\frac{1}{m_rn_r}F(r)_{m_rn_r})^{d_r}\right),\\
B=\left((-\frac{1}{m_1n_1+i_1}F(1)_{m_1n_1+i_1})^{d_1}\cdot\ldots\cdot (-\frac{1}{m_rn_r+i_r}F(r)_{m_rn_r+i_r})^{d_r}\right).
\end{array}
$$
Then
$$
c-\epsilon =-\frac{\epsilon}{2}+(c-\frac{\epsilon}{2})\le Q+A\le B
$$
where the first inequality is by (\ref{eq82'}) and (\ref{eq71}) and the second inequality is by (\ref{eq70})
 if $n_j>n^*$ for all $j$, and $1\le i_j<m_j$ for $1\le j\le r$, showing the existence of the limit of Theorem \ref{Theorem70}.
\end{proof}

In general, we do not have the desirable condition that $-\frac{1}{m}F_ m\le -\frac{1}{m+1}F_{m+1}$, which would simplify the calculations of the proof of Theorem \ref{Theorem70}. We give a simple example of a graded filtration $\mathcal I=\{I_n\}$ of $m_R$-primary ideals where this condition fails. 

For $m\in \ZZ_{>0}$, define
$$
\sigma(m)=\left\{\begin{array}{ll} \frac{m}{2}&\mbox{ if $m$ is even}\\
\frac{m+1}{2}&\mbox{ if $m$ is odd}.\\
\end{array}\right.
$$
for $m,n\in\ZZ_{>0}$, we have that
$\sigma(m)+\sigma(n)\ge\sigma(m+n)$
and
$\sigma(m+1)\ge \sigma(m)$.
Let $R$ be a one dimensional regular local ring, and $t$ be a generator of $m_R$. Let $\mathcal I=\{I_n\}$ where  $I_n=(t^{\sigma(n)})$, an ideal in $R$.
$\mathcal I$ is a graded filtration. If $m$ is even and positive, then 
$$
\sigma(m)(m+1)=\frac{m}{2}(m+1)<(\frac{m}{2}+1)m=\sigma(m+1)m.
$$
Thus $I_m^{m+1}\not\subset I_{m+1}^m$ if $m$ is even and positive. Since $R$ is regular of dimension 1, $\mbox{BirMod}(R)=\{\mbox{Spec}(R)\}$,
so $\mathcal O_{Y_m}(-F_m)=I_m$  for all $m$. Thus $-\frac{1}{m}F_m\not\le -\frac{1}{m+1}F_{m+1}$  if $m$ is even.

\section{Multiplicities of graded $m_R$-families}\label{SecMult}

We give a  proof of the existence of the multiplicity $e(\mathcal I)$ of a graded $m_R$-family which does not rely on the theory of Okounkov bodies, but only uses intersection theory.

\begin{Theorem}\label{Theorem5}(Theorem \ref{multlim} in the introduction)
Let $R$ be a $d$-dimensional  Noetherian local ring and $\mathcal I=\{I_n\}$ be a graded $m_R$-family. Then
$$
\lim_{n\rightarrow\infty}\frac{e(I_n)}{n^d}
$$
exists.
\end{Theorem}

\begin{proof} If $d=0$ then $e(R/I_n)=e(R)$ for all $n$ so that $e(\mathcal I)=e(I)$. Suppose that $d\ge 1$.
Take $\mathcal I(1)=\mathcal I$ and $r=1$ in
 Theorem \ref{Theorem70}, to see that the limit 
 $$
 ((\mathcal I)^d)=\lim_{n\rightarrow \infty}\left(\left(\frac{-F_n}{n}\right)^d\right)
 $$ exists. 
The theorem now follows from 
the identity 
$$
e(I_n)=-\left((I_n\mathcal O_{Y_n})^d\right)=-\left((-F_n)^d\right)
$$
 of Theorem \ref{Theorem20}.
 \end{proof}

\begin{Corollary}\label{Cor73} Let $R$ be a $d$-dimensional Noetherian local ring, $\mathcal I=\{I_n\}$ be a graded $m_R$-family and let $M$ be a finitely generated $R$-module. Let $\Lambda$ be the set of minimal prime ideals $P$ of $R$ such that $\dim R/P=d$. For $P\in \Lambda$, let
$\mathcal I(R/P)=\{I_n(R/P)\}$, a graded family of $m_R$-primary ideals on $R/P$.
Then
$$
\lim_{n\rightarrow\infty}\frac{e(I_n;M)}{n^d}=\sum_{P\in\Lambda}e(\mathcal I(R/P))\ell_{R_P}(M_P).
$$ 
In particular, the limit $\lim_{n\rightarrow\infty}\frac{e(I_n;M)}{n^d}$ exists.
\end{Corollary}

\begin{proof} By the associativity formula for multiplicities, Theorem 11.2.4 \cite{HS},
$$
e(I_n;M)=\sum_{P\in \Lambda}e(I_n(R/P))\ell_{R/P}(M_P)
$$
for all $n$. Since  the limit
$$
e(\mathcal I(R/P))=\lim_{n\rightarrow\infty}\frac{e(I_n(R/P))}{n^d}
$$
exists
by Theorem \ref{Theorem5}, we have that the desired identity holds.
\end{proof}

\vskip .2truein

The twist $\mathcal I^{(c)}$ of a filtration $\mathcal I$ by $c\in \RR_{>0}$ is defined in equation (\ref{eqIn3}) of the introduction.
 
  \begin{Lemma}\label{Prop71}
  Let $\mathcal I(1),\ldots,\mathcal I(d)$ be graded $m_R$-families and $c_1,\ldots,c_d\in \RR_{>0}$. Then
  $$
  \left((\mathcal I(1)^{(c_1)})\cdot\ldots\cdot (\mathcal I(d)^{(c_d})\right) =
  c_1c_2\cdots c_d\left(\mathcal I(1)\cdot\ldots\cdot \mathcal I(d)\right).
  $$
  \end{Lemma}
  
  \begin{proof} By Theorem \ref{Theorem70},
  $$
  \begin{array}{l}
  \displaystyle\left((\mathcal I(1)^{(c_1)})\cdot\ldots\cdot (\mathcal I(d)^{(c_d)})\right)   
  =\lim_{n_1,\ldots,n_d\rightarrow\infty}\left((\frac{-F(1)_{\lceil c_1n_1\rceil}}{n_1})\cdot\ldots\cdot (\frac{-F(d)_{\lceil c_dn_d\rceil}}{n_d})\right)\\
  =\displaystyle\lim_{n_1,\ldots,n_d\rightarrow\infty}\prod_{j=1}^d\frac{\lceil c_jn_j\rceil}{n_j}\left((\frac{-F(1)_{\lceil c_1n_1\rceil}}{\lceil c_1n_1\rceil})\cdot\ldots\cdot
  (\frac{-F(d)_{\lceil c_dn_d\rceil}}{\lceil c_dn_d\rceil})\right)\\
   = c_1c_2\cdots c_d\left(\mathcal I(1)\cdot\ldots\cdot \mathcal I(d)\right).
       \end{array}
  $$

  \end{proof}
  
  The product of graded families $\mathcal I(1)\mathcal I(2)\cdots\mathcal I(r)$ is defined in equation (\ref{eqIn4}) of the introduction.  
 We now give the proof of Theorem \ref{Minkcoef*} and formulas (\ref{MixInt}) and (\ref{eqIn15}) of the introduction, establishing mixed multiplicities of graded $m_R$-families.



 Let notation be as in Theorem \ref{Theorem70}.
$$
\begin{array}{l}
\displaystyle\lim_{m\rightarrow\infty} \frac{e(I(1)_{mn_1}\cdots\ I(r)_{mn_r})}{m^d}=
\lim_{m\rightarrow\infty}-\left(\left(\frac{-F(1)_{mn_1}}{m}-\frac{-F(2)_{mn_2}}{m}+\cdots+\frac{-F(r)_{mn_r}}{m}\right)^d\right)\\
= \displaystyle\lim_{m\rightarrow\infty}\left(\sum_{d_1+\ldots+d_r=d}\frac{-d!}{d_1!\cdots d_r!}\left(\left(\frac{-F(1)_{mn_1}}{m}\right)^{d_1}\cdot\ldots\cdot
\left(\frac{-F(r)_{mn_r}}{m}\right)^{d_r}\right)\right)\\
= \displaystyle\lim_{m\rightarrow\infty}\left(\sum_{d_1+\ldots+d_r=d}\frac{-d!}{d_1!\cdots d_r!}\left(\left(\frac{-F(1)_{mn_1}}{mn_1}\right)^{d_1}\cdot\ldots\cdot
\left(\frac{-F(r)_{mn_r}}{mn_r}\right)^{d_r}\right)n_1^{d_1}\cdots n_r^{d_r}\right)\\
= \displaystyle\sum_{d_1+\ldots+d_r=d}\frac{-d!}{d_1!\cdots d_r!} \left(\lim_{m\rightarrow\infty}\left(\left(\frac{-F(1)_{mn_1}}{mn_1}\right)^{d_1}\cdot\ldots\cdot
\left(\frac{-F(r)_{mn_r}}{mn_r}\right)^{d_r}\right)\right)n_1^{d_1}\cdots n_r^{d_r}\\
= \displaystyle\sum_{d_1+\ldots+d_r=d}\frac{-d!}{d_1!\cdots d_r!} \left(\lim_{m\rightarrow\infty}\left(\left(\frac{-F(1)_{m}}{m}\right)^{d_1}\cdot\ldots\cdot
\left(\frac{-F(r)_{m}}{m}\right)^{d_r}\right)\right)n_1^{d_1}\cdots n_r^{d_r}
\end{array}
$$
by multilinearity of intersection products and Theorem \ref{Theorem70}. 

Thus Theorem \ref{Minkcoef*} and formula (\ref{MixInt}) are true. Formula (\ref{eqIn15}) now follows from Theorem \ref{MixMultIdeal}.
We also establish mixed multiplicities of finitely generated modules with respect to graded $m_R$-families. 
 
\begin{Theorem}\label{MinkcoefM} Let $R$ be a $d$-dimensional noetherian local ring, $\mathcal I(1)=\{I(1)_n\},\ldots, \mathcal I(r)=\{I(r)_n\}$ be graded $m_R$-families and $M$ be a finitely generated $R$-module. Then 
$$
\lim_{m\rightarrow\infty} \frac{e(I(1)_{mn_1}\cdots\ I(r)_{mn_r};M)}{m^d}
$$
is a homogeneous polynomial of degree $d$ in $n_1,\ldots,n_r$.
\end{Theorem}
The coefficient of $n_1^{d_1}\cdots n_r^{d_r}$ of the above polynomial is the  mixed multiplicity 
$$
e(\mathcal I(1)^{[d_1]},\ldots,\mathcal I(r)^{[d_r]};M)
$$ 
of $M$.

\begin{proof} By Corollary \ref{Cor73},
$$
\begin{array}{lll}
\displaystyle\lim_{m\rightarrow\infty} \frac{e(I(1)_{mn_1}\cdots\ I(r)_{mn_r};M)}{m^d}&=&\displaystyle \lim_{m\rightarrow\infty}\frac{1}{m^d}\left(
\sum \ell_{R_P}(M_P)(e(I(1)_{mn_1}\cdots\ I(r)_{mn_r}R/P)\right)\\
&=& \displaystyle\sum \ell_{R_P}(M_P)\left(\lim_{m\rightarrow\infty}\frac{e(I(1)_{mn_1}\cdots\ I(r)_{mn_r}R/P)}{m^d}\right)
\end{array}
$$
where the sums are over the minimal primes $P$ of $R$ such that $\dim R/P=d$. The resulting expression is a homogeneous polynomial of degree $d$ by Theorem \ref{Minkcoef*}.
\end{proof}

\section{A comparison of volume and multiplicity on Noetherian local rings}\label{SecVolMult}

\begin{Proposition}\label{PropVUB} Suppose that $R$ is a $d$-dimensional Noetherian local ring, $\mathcal I=\{I_n\}$ is a graded $m_R$-family and $M$ is a finitely generated $R$-module. Then
$$
{\rm vol}(\mathcal I;M)\le e(\mathcal I;M).
$$
\end{Proposition}

\begin{proof} Volume and multiplicity are preserved if we replace $R$ with its $m_R$-adic completion $\hat R$,  $\mathcal I$ with $\hat{\mathcal I}=\{\hat{I_n}\}$ and $M$ with $\hat M$. Thus we may assume that $R$ is complete. 
Let 
$$
0=M_0\subset M_1\subset \cdots \subset M_n=M
$$
be a prime filtration of $M$, so that $M_{i+1}/M_i\cong R/P_i$ where $P_i$ is a prime ideal of $R$ for $0\le i\le n-1$. We then have short exact sequences
\begin{equation}\label{eqUB2}
\begin{array}{l}
0\rightarrow M_i\rightarrow M_{i+1}\rightarrow R/P_{i+1}\rightarrow 0\mbox{ for $2\le i\le n-1$ and}\\
0\rightarrow R/P_1\rightarrow M_2\rightarrow R/P_2\rightarrow 0.
\end{array}
\end{equation}
Tensoring (\ref{eqUB2}) with $R/I_n$, we have right exact sequences
$$
\begin{array}{l}
M_i/I_nM_i\rightarrow M_{i+1}/I_nM_{i+1}\rightarrow (R/P_{i+1})/I_n(R/P_{i+1})\rightarrow 0\mbox{ for $2\le i\le n-1$ and}\\
(R/P_1)/I_n(R/P_1)\rightarrow M_2/I_nM_2\rightarrow (R/P_2)/I_n(R/P_2)\rightarrow 0.
\end{array}
$$
Thus
$$
\ell(M/I_nM)\le \sum_{i=1}^n \ell((R/P_i)/I_n(R/P_i)).
$$
Let $\{Q_1,\ldots, Q_t\}$ be the distinct prime ideals appearing in  (\ref{eqUB2}), ordered so that $\dim R/Q_i=d$ for $1\le i\le s$ and $\dim R/Q_i<d$ for $s<i\le t$.
Let $a_i$ be the number of times $R/Q_i$ appears in \ref{eqUB2}. Then 

\begin{equation}\label{eqUB5}
\begin{array}{lll}
\ell(M/I_nM)&\le& \sum_{i=1}^ta_i\ell((R/Q_i)/I_n(R/Q_i))\\
&=&\sum_{i=1}^s\ell_{R_{Q_i}}(M_{Q_i})\ell((R/Q_i)/I_n(R/Q_i))
+\sum_{i=s+1}^ta_i\ell((R/Q_i)/I_n(R/Q_i))
\end{array}
\end{equation}
where the equality on the second line is obtained by localizing at $Q_i$ for $1\le i\le s$. Thus
$$
\begin{array}{l}
\displaystyle{\rm vol}(\mathcal I;M)=\limsup_{n\rightarrow\infty}d!\frac{\ell(M/I_nM)}{n^d}\\
\le 
\displaystyle\limsup_{n\rightarrow\infty}\frac{d!}{n^d}\left[\sum_{i=1}^s\ell_{R_{Q_i}}(M_{Q_i})\ell((R/Q_i)/I_n(R/Q_i))
+\sum_{i=s+1}^ta_i\ell((R/Q_i)/I_n(R/Q_i))\right]\\
\le \displaystyle
\sum_{i=1}^sd!\ell_{R_{Q_i}}(M_{Q_i})\limsup_{n\rightarrow\infty}\frac{\ell((R/Q_i)/I_n(R/Q_i))}{n^d}
+\sum_{i=s+1}^td!a_i\limsup_{n\rightarrow\infty}\frac{\ell((R/Q_i)/I_n(R/Q_i))}{n^d}\\
=\sum_{i=1}^s\ell_{R_{Q_i}}(M_{Q_i}){\rm vol}(\mathcal I(R/Q_i)) = \sum_{i=1}^s\ell_{R_{Q_i}}(M_{Q_i})e(\mathcal I(R/Q_i)) \\
=e(\mathcal I;M)
\end{array}
$$
where the inequality on the second line is by (\ref{eqUB5}),  the second equality on the fourth line is by Theorem \ref{volmult} and the equality on the last line is by Corollary \ref{Cor73}.
\end{proof}

It is possible for the volume to be strictly less than the multiplicity, as is shown in the following example. 

\begin{Example}\label{VolEx} There exists a Noetherian local ring $R$ and a graded $m_R$-filtration $\mathcal J$ on $R$ such that
 $$
 {\rm vol}(\mathcal J)\ne e(\mathcal J).
 $$
  In this example, we necessarily have that $\dim N(\hat R)=\dim R$. Further, ${\rm vol}(\mathcal J)$ does not exist as a limit, but only as a limsup.
 \end{Example}
 
\begin{proof}
 In Example 5.3 \cite{C2},  a graded $m_R$-filtration $\mathcal I$ such that $\mbox{vol}(\mathcal I)$ does not exist as a limit is constructed. The example does satisfy the good condition that volume = multiplicity holds; that is, ${\rm vol}(\mathcal I)=e(\mathcal I)$. We modify this example to find a filtration $\mathcal J$ such that ${\rm vol}(\mathcal J)\ne e(\mathcal J)$, and volume does not exist as a limit.

Let $k$ be a field and $k[[x_1,\ldots,x_d,y]]$ be a power series ring over $k$. Let 
$$
R=k[[x_1,\ldots,x_d,y]]/(y^2).
$$
 Let $\overline x_1,\ldots,\overline x_d,\overline y$ be the respective classes of $x_1,\ldots,x_d,y$ in $R$. In equation (22) of \cite{C2}, a function $\sigma(n)$ is defined which has the properties that
$\sigma(m)\le\sigma(n)$ for all $m\le n$, $0\le \sigma(n)\le \frac{n}{2}$ for all $n$ and there are arbitrarily large $n$ such that $\frac{\sigma(n)}{n}$ comes arbitrarily close to 0 and there are arbitrarily large $n$ such that $\frac{\sigma(n)}{n}$ comes arbitrarily closed to $\frac{1}{2}$. 

Let $N_i$ be the set of monomials of degree $i$ in $\overline x_1,\ldots,\overline x_d$. Define $m_R$-primary ideals $J_n$ in $R$ by 
$$
J_n=(N_n,\overline y N_{n-\sigma(n)-\lceil \frac{n}{4}\rceil})\mbox{ for }n\ge 1
$$
and $J_0=R$. $\mathcal J=\{J_n\}$ is a graded $m_R$-filtration since $\sigma(m)\le \sigma(n)$ for $m\le n$. $R/J_n$ has a $k$-basis consisting of 
$$
\{N_i\mid i<n\}\mbox{ and }\{\overline yN_j\mid j<n-\sigma(n)-\lceil \frac{n}{4}\rceil\}.
$$
Thus
$$
\ell(R/J_n)=\binom{n+d-1}{d}+\binom{n-\sigma(n)-\lceil \frac{n}{4}\rceil+d-1}{d}
$$
for all $n\ge 1$, so
$$
\begin{array}{lll}
{\rm vol}(\mathcal J)&=&
\displaystyle\limsup_{n\rightarrow \infty}d!\frac{\ell(R/J_n)}{n^d}=\limsup_{n\rightarrow\infty}\left(\frac{n^d+(n-\sigma(n)-\lceil \frac{n}{4}\rceil)^d}{n^d}\right)\\
&=& \displaystyle\limsup_{n\rightarrow\infty} \left(1+(1-\frac{\sigma(n)}{n}-\lceil\frac{n}{4}\rceil\frac{1}{n})^d\right)= 1+\left(\frac{3}{4}\right)^d<2.
\end{array}
$$
Let $S=k[[\overline x_1,\ldots,\overline x_d]]$. By the associativity formula for multiplicity, 
$$
e(J_n)=2e(J_nS)=2e((m_S)^nS)=2n^de(m_S)=2n^d.
$$
Thus $e(\mathcal J)=2>\mbox{vol}(\mathcal J)$.
\end{proof}

\section{valuations, integral closure and valuation saturation}\label{Vals} Divisorial valuations, $m_R$-valuations and the set  $V(R)$ of $m_R$-valuations  are defined in the introduction.



We first prove the following lemma which was referenced  in the introduction.

\begin{Lemma}\label{LemVal} Suppose that $R$ is a $d$-dimensional Noetherian local ring. Then there is a natural 1-1 correspondence between $V(R)$ and the disjoint union of $V(\hat R/P_i)$ where $P_i$ are the minimal prime ideals of $\hat R$ such that $\dim \hat R/P_i=d$.
\end{Lemma}

\begin{proof} We will use the following observation in our proof.  Suppose that $R$ is a local domain and $P$ is a prime ideal of $\hat R$ such that $P\cap R=0$. Let $\omega$ be a rank 1 valuation of the quotient field of $\hat R/P$ and let $v$ be its restriction to $R$. Given $g\in \hat R/P$ and $n>0$, there exists $f\in R$ and $h\in m_R^n\hat R/P$ such that $g=f+h$. Taking $n$ so large that $nv(m_R)>w(g)$, we have that $v(f)=v(g)$.  Thus the valuations $v$ and $w$ have the same residue fields and value groups.

Suppose that $v\in V(R)$ Then there exists a minimal prime $Q$ of $R$ such that $v$ is a valuation of the quotient field of $R/Q$ which dominates $R/Q$ and whose residue field has transcendence degree $d-1$ over $R/m_R$. $v$ extends uniquely to a valuation $\hat v$ of $\hat R$ by defining $\hat v(\sigma)=
\displaystyle\lim_{n\rightarrow \infty}v(\sigma_n)\in\ZZ\cup\{\infty\}$ if $\sigma\in \hat R$ and $\{\sigma_n\}$ is a Cauchy sequence in $R$ which converges to $\sigma$.
Let $P=\{f\in \hat R\mid v(f)=\infty\}$. Then $P$ is a prime ideal of $\hat R$ and $\hat v$ is a valuation of the quotient field of $\hat R/P$ such that the value group of $\hat v$ is $\ZZ$ and the residue field of $\hat v$ is equal to the residue field of $v$. By Abhyankar's lemma (Theorem 6.6.7 \cite{HS})
$$
d=\mbox{ratrank }\hat v+\mbox{trdeg}_{R/m_R}\hat v\le \dim \hat R/P\le d.
$$
Thus $\dim \hat R/P=d$ and  $\hat v\in V(\hat R/P)$.

Suppose that $P$ is a minimal prime ideal of $\hat R$ such that $\dim \hat R/P=d$ and $w\in V(\hat R/P)$. Let $Q=P\cap R$ and $w^*$ be the restriction of $w$ to the quotient field of $R/Q$. As commented at the beginning of this proof, the value group of $\omega^*$ is $\ZZ$ and the residue field of $w^*$ has transcendence degree $d-1$ over $R/m_R$. Thus $w^*\in V(R)$. 

We have that $(\hat v)^*=v$ and $\widehat{(w^*)}=w$ for the correspondences $v\mapsto \hat v$ and $w\mapsto w^*$ defined above and the lemma follows.
\end{proof}






We will need the following identities on the twist of a filtration (defined in equation (\ref{eqIn3}) of the introduction).
\begin{equation}\label{eq51}
v(\mathcal I^{(c)})= \inf \frac{v(I_{\lceil cm\rceil})}{m}=\inf \frac{\lceil c m\rceil}{m}\frac{v(I_{\lceil c m\rceil})}{\lceil c m\rceil}
=c v(\mathcal I)
\end{equation}
and
\begin{equation}\label{eq50}
e(\mathcal I^{(c)})=\lim_{m\rightarrow\infty}\frac{e(I_{\lceil cm\rceil})}{m^d}=\lim_{m\rightarrow\infty}\left(\frac{\lceil c m\rceil}{m}\right)^d\frac{e(I_{\lceil cm\rceil})}{\lceil c m\rceil^d}=c^d\left(\lim_{n\rightarrow\infty}\frac{e(I_n)}{n^d}\right)= c^d e(\mathcal I).
\end{equation}

It follows from (\ref{eq51}) that
\begin{equation}\label{eq88}
\widetilde{\mathcal I^{(c)}}=(\tilde{\mathcal I})^{(c)}
\end{equation}
for $c\in \RR_{\ge 0}$.

We further have that 
\begin{equation}\label{eq89}
\overline{\mathcal I^{(c)}}=(\overline{\mathcal I})^{(c)}.
\end{equation}







The integral closure $\overline{\mathcal I}$ of a filtration and the saturation $\tilde{\mathcal I}$ of a saturation are defined in the introduction. We have that
$\mathcal I \subset \overline{\mathcal I}\subset \tilde{\mathcal I}$. 


The following lemma is referenced in the introduction. 

\begin{Lemma}\label{SatInt} The graded $R$-algebra $R[\tilde{\mathcal I}]=\sum_{n\ge 0}\tilde I_nt^n$ is integrally closed in $R[t]$.
\end{Lemma}

\begin{proof} The integral closure of $R[\tilde{\mathcal I}]$ in $R[t]$ is graded by Theorem 2.3.2 \cite{HS}. So it suffices to show that if $f\in R$,  $s\in \ZZ_{>0}$, and $ft^s$ is integral over $R[\tilde{\mathcal I}]$, then $f\in \tilde I_s$. Suppose that $ft^s$ is integral over $R[\tilde{\mathcal I]}$. Then there exists a relation
$$
f^n+a_1f^{n-1}+\cdots+a_n=0
$$
where $a_i\in \tilde I_{is}$ for all $i$. Let $v\in V(R)$, so that $v\in V(\hat R/P)$ for some  minimal prime $P$ of $\hat R$ such that $\dim \hat R/P=d$.
If $v(f)=\infty$, that is, $f\in P\cap R$, then certainly $v(f)\ge sv(\mathcal I)$. If $v(f)<\infty$, then 
$v(f^n)\ge \min\{v(a_if^{n-i})\mid 1\le i\le n\}$ since 
$$
v(f^n+a_1f^{n-1}+\cdots+a_n)=v(0)=\infty.
$$
 Thus there exists an $1\le i$ such that 
$nv(f)\ge v(a_i)+(n-i)v(f)$, so that 
$$
iv(f)\ge v(a_i)\ge is v(\mathcal I).
$$
 Thus $v(f)\ge sv(\mathcal I)$. Since this relation holds for all $v\in V(R)$, $f\in \tilde I_s$, and so $ft^s\in R[\tilde{\mathcal I}]$.
\end{proof}

\begin{Remark} $v(\mathcal I)=0$ for all $v\in V(R)$ if and only if $\tilde{\mathcal I}=\{J_n\}$ where $J_n=m_R$ for all $n>0$.
\end{Remark}

\begin{Remark} $v(\mathcal I)\le v(\mathcal J)$ for all $v\in V(R)$ if and only if $\tilde{\mathcal J}\subset \tilde{\mathcal I}$
\end{Remark}



\section{A saturated filtration which is not real divisorial}\label{Notdivex}

In Subsection \ref{SubSecVal} of the introduction, we pointed out that for an $I$-adic filtration $\mathcal I=\{I^n\}$, where $I$ is an $m_R$-primary ideal in a local ring $R$ which is not formally equidimensional, it may be that $\tilde{\mathcal I}$ is not a divisorial filtration. It is more difficult to find an example of a saturated graded $m_R$-family on a formally equidimensional local ring which is not real divisorial.

 We give an example in this section of a saturated filtration on a two dimensional regular local ring which is not real divisorial. 
 We show that an example from \cite{C10} satisfies this condition. 
In \cite{C10} a graded filtration $\mathcal I=\{I_n\}$ on a two dimensional regular local ring is constructed which has the property that its fiber cone 
$\oplus_{n\ge 0}I_n/m_RI_n$ has infinitely many minimal primes. As a consequence, the $R$-scheme $\mbox{Proj}(\oplus_{n\ge 0}I_n/m_RI_n)$ is not a Noetherian scheme. We will show that this filtration is saturated, and that it is not a real divisorial filtration.

  We begin by reviewing  the construction of $\mathcal I$ in \cite{C10}.
Let $R$ be a 2 dimensional regular local ring with algebraically closed residue field. 
 Let $X_0=\mbox{Spec}(R)$ and $X_1$ be the blowup of the maximal ideal $m_R$ of $R$, with integral exceptional divisor $E(1)_1$. Let $p_1\in X_1$ be a closed point on $E(1)_1$. Let $X_2\rightarrow X_1$ be the blowup of $p_1$. Let $E(2)_1$ be the strict transform of $E(1)_1$ on $X_2$ and $E(2)_2$ be the integral exceptional divisor mapping to $p_1$.
Let $X_3\rightarrow X_2$ be the blowup of a closed point $p_2\in E(2)_2\setminus E(2)_1$. Let $E(3)_3$ be the integral exceptional divisor mapping to $p_2$. Let $E(3)_1$ be the strict transform of $E(2)_1$ and $E(3)_2$ be the strict transform of $E(2)_2$.
Inductively continue this construction, letting $X_l\rightarrow X_{l-1}$ be the blowup of a point $p_{l-1}\in E(l-1)_{l-1}\setminus E(l-1)_{l-2}$. Let $E(l)_l$ be the integral exceptional divisor mapping to $p_{l-1}$ and $E(l)_i$ be the strict transform of $E(l-1)_i$ for $i<l$. Then $X_l$ is nonsingular, and the reduced exceptional divisor of $X_l\rightarrow \mbox{Spec}(R)$ is the divisor
$E(l)_1+E(l)_2+\cdots+E(l)_l$. 
When $l$ is understood,  we will often write $E_i$ for $E(l)_i$. In particular, the valuation $v_{E_i}$ is well defined. 

In Section 4 \cite{C10},  the graded filtration $\mathcal I=\{I_m\}$ is defined  by 
$$
I_m=\{f\in R\mid v_{E_i}(f)\ge \frac{m(2^i-1)}{2^{i-1}}\mbox{ for }i\ge 1\}.
$$
It is shown in Proposition 4.1, the discussion after it, and the beginning of  Section 4.2 of \cite{C10} that 
\begin{equation}\label{Ex1}
I_m\mathcal O_{X_l}\mbox{ is invertible for $m<2^{l-1}$}
\end{equation}
and
\begin{equation}\label{Ex2}
v_{E_n}(I_m)=\lceil\frac{m(2^n-1)}{2^{n-1}}\rceil\mbox{ for all $m$ and $n$}.
\end{equation}
From (\ref{Ex2}) we compute that
$$
v_{E_i}(\mathcal I)=\inf_m\frac{v_{E_i}(I_m)}{m}
=\frac{2^i-1}{2^{i-1}}\mbox{ for all $i$}.
$$
We now verify that $\mathcal I$ is saturated. By definition, 
$$
\tilde I_n=\{f\in R\mid w(f)\ge nw(\mathcal I)\mbox{ for all }w\in V(R)\}.
$$
If $f\in \tilde I_n$, then certainly $v_{E_i}(f)\ge nv_{E_i}(\mathcal I)$ for all $v_{E_i}$, so $f\in I_n$. 
 
 Let $w\in V(R)$. If the center of $w$ on $X_l$ is on $E(l)_l$ for all $l$, then the center of $w$ on $X_l$ must be the  point $p_l$ for all $l$, so that $\mathcal O_w$ dominates $\mathcal O_{X_l,p_l}$ for all $l$. We have that $\mbox{trdeg}_{R/m_R}\mathcal O_w/m_w=1$ since $w$ is an $m_R$-valuation. Thus there exists $t\in  \mathcal O_w$ such that the residue $\overline t$ of $t$ in $\mathcal O_w/m_w$ is transcendental over $R/m_R=k$. Write $t=\frac{f}{g}$ with $f,g\in R$. There exists $\overline n$ such that $(f,g)\mathcal O_{X_{\overline n},p_{\overline n}}$ is a principal ideal (for instance by Theorem 4.11 \cite{ResS}). Thus $f$ divides $g$ or $g$ divides $f$ in $\mathcal O_{X_{\overline n},p_{\overline n}}$, so that $t$ or $\frac{1}{t}$ is in $\mathcal O_{X_{\overline n},p_{\overline n}}$.   Thus $\overline t\in \mathcal O_{X_{\overline n},p_{\overline n}}/m_{p_{\overline n}}$ a contradiction, since  $\mathcal O_{X_{\overline n},p_{\overline n}}/m_{p_{\overline n}}=k=R/m_R$,
 giving a contradiction. Thus there exists $l_0$ such that if $l\ge l_0$ then the center of $w$ on $X_l$ is on $E(l)_r$ with $r<l$. We have that $X_m\rightarrow X_{l_0}$ is an isomorphism for $m\ge l_0$ in a neighborhood of $E_r$. Thus letting $p$ be the center of $w$ on $X_{l_0}$, we have that  
 \begin{equation}\label{Ex3}
 I_m\mathcal O_{X_{l_0},p}\mbox{ is invertible for all $m\ge l_0$.}
 \end{equation}
 We have four possible cases for the center $p$.
 \begin{enumerate}
 \item[1)] $p=E_r$.
 \item[2)] $p\in E_r\setminus (E_{r-1}\cup E_{r+1})$.
 \item[3)] $p=E_r\cap E_{r+1}$.
 \item[4)] $p=E_{r-1}\cap E_r$.
 \end{enumerate}
 
 In Case 1), the divisorial valuation $w$ is equal to $v_{E_r}$, so that $w(\mathcal I)\in \QQ$. Suppose that Case 2) holds. Let $f_r$ be a local equation of $E_r$ in $\mathcal O_{X_{l_0},p}$.
 By (\ref{Ex1}), (\ref{Ex2}) and (\ref{Ex3}), we have that 
 $$
 I_m\mathcal O_{X_{l_0},p}=f_r^{\lceil\frac{m(2^r-1)}{2^{r-1}}\rceil}\mathcal O_{X_{l_0},p}
 $$
 for all $m$. Thus 
 $$
 w(I_m)=\lceil\frac{ m(2^r-1)}{2^{r-1}}\rceil w(f_r)
 $$
 for all $m$, and so
 $$
 w(\mathcal I)=\frac{(2^r-1)}{2^{r-1}}w(f_r)\in\QQ.
 $$
 In Case 3), let $f_r$ be a local equation of $E_r$ in $\mathcal O_{X_{l_0},p}$ and $f_{r+1}$ be a local equation of $E_{r+1}$. Then by (\ref{Ex1}), (\ref{Ex2}) and (\ref{Ex3}), we have that 
 $$
 \omega(I_m)=\lceil\frac{ m(2^r-1)}{2^{r-1}}\rceil w(f_r)+ \lceil\frac{m(2^{r+1}-1)}{2^{r}}\rceil w(f_{r+1}) 
 $$
 for all $m$, so that 
 $$
 w(\mathcal I)=\frac{(2^r-1)}{2^{r-1}}w(f_r)+\frac{(2^{r+1}-1)}{2^{r}}w(f_{r+1}) \in\QQ.
 $$
 Case 4) is computed as in Case 3), and we conclude  again that $w(\mathcal I)\in \QQ$.
 Thus we have that 
 \begin{equation}\label{Ex4}
 w(\mathcal I)\in \QQ\mbox{ for all }\omega\in V(R).
 \end{equation} 
 Suppose that $\mathcal I$ is a real divisorial filtration. Then there exist $w_1,\ldots, w_s\in V(R)$ and $a_1,\ldots, a_s\in \RR_{>0}$ such that
 $$
 I_m=\{f\in R\mid w_i(f)\ge a_im\mbox{ for }1\le i\le s\}.
 $$ 
 For fixed $i$, given $\epsilon>0$, there exists $m$ and $f\in I_m$ such that $\frac{w_i(f)}{m}<w_i(\mathcal I)+\epsilon$. $\frac{w_i(f)}{m}\ge a_i$ implies 
  $a_i\le w_i(\mathcal I)$ for all $i$, so that 
 $$
 I_m=\{f\in R\mid w_i(f)\ge w_i(\mathcal I)m\mbox{ for }1\le i\le s\}
 $$
 so that $\mathcal I$ is a $\QQ$-divisorial filtration. It is shown in \cite{Ru} or \cite{C11} that this condition implies that $\mbox{Proj}(\oplus_{n\ge 0}I_n)$ is a Noetherian scheme. But this contradicts the fact that $\mbox{Proj}(\oplus_{n\ge 0}I_n)$ is not Noetherian, as is shown in Section 4 of \cite{C10} since the fiber cone $\oplus_{n\ge 0}I_n/m_RI_n$ has infinitely many primes. 
 Thus the saturated graded $m_R$-filtration $\mathcal I$ is not real divisorial.

\section{Rees's Theorem}\label{ReesSec}

In this section we extend a famous theorem of Rees to graded $m_R$-families on a Noetherian local ring. Background, discussion  and some history of this problem are given in SubSection \ref{ReesSubSec}.
 Theorem \ref{Theorem14} (Theorem \ref{Theorem14*} of the introduction) is proven in \cite{BLQ} when $R$ is analytically irreducible. 
 In \cite{BLQ}, they prove Theorem \ref{Theorem14} for  the volume $\mbox{vol}(\mathcal I)$ instead of for  multiplicity $e(\mathcal I)$. 
 Since they are in an analytically irreducible local ring,  the volume = multiplicity formula
$$
\mbox{\rm vol}(\mathcal I)=\lim_{p\rightarrow\infty}\frac{e(I_p)}{p^d}=e(\mathcal I)
$$
of Theorem \ref{volmult}  holds, so their theorem implies Theorem \ref{Theorem14} for analytically irreducible local rings. It is then not difficult to extend Theorem \ref{Theorem14}  to arbitrary  local rings.

The proof in \cite{BLQ} is dependent on Okounkov bodies which establishes the existence of the limits (\ref{eq10}) and (\ref{eq33}). 
In this section, we modify their proof for analytically irreducible local rings to give a proof that is independent of the theory of volume and of Okounkov bodies.
That is, we replace calculations of ${\rm vol}(\mathcal I)$ with calculations of $e(\mathcal I)$.

The following lemma is proven for analytically irreducible local rings in Proposition 2.7 \cite{BLQ}. We use the idea of their proof to give a proof which does not rely on the fact that volumes are limits on analytically irreducible local rings. The graded filtration $\mathcal I(v)$ associated to $v\in V(R)$ is defined in (\ref{eqvalfin}).

\begin{Lemma}\label{Lemma21} Suppose that $R$ is a $d$-dimensional Noetherian local ring and $v\in V(R)$. Then $e(\mathcal I(v))>0$.
\end{Lemma}

\begin{proof} There exists by Lemma \ref{LemVal} a minimal prime $P$ of $\hat R$ such that $\dim \hat R/P=d$ and $v\in V(\hat R/P)$. Let $S=\hat R/P$ and 
$J(v)_n=\{f\in S\mid v(f)\ge n\}$ for $n\in \ZZ_{\ge 0}$, so that  $J(v)_n=I(v)_nS$.
We have that
\begin{equation}\label{eq90} 
\ell(R/(I(v)_n)^m)=\ell(\hat R/(I(v)_n)^m\hat R)\ge  \ell(S/(J(v)_n)^mS).
\end{equation}
Let $v_1,\ldots,v_r$ be the Rees valuations of $m_S$, so that there exist $a_1,\ldots,a_r\in \ZZ_{>0}$ such that 
$$
I(v_1)_{a_1n}\cap \cdots \cap I(v_r)_{a_rn}=\overline{m_S^n}
$$
for all $n\in \ZZ_{>0}$. Since $S$ is analytically irreducible, the $v_i$ are all $m_S$-valuations (Lemma \ref{LemVal}), and Izumi's lemma holds (Theorem E, page 409 \cite{R3}, or \cite{RoSp}). Thus there exists $\gamma\in \QQ_{>0}$ such that
$J(v)_n\subset I(v_i)_{\lceil n\gamma a_i\rceil}$ for all $n\in \ZZ_{>0}$, and so
$$
J(v)_n\subset I(v_1)_{\lceil n\gamma a_1\rceil}\cap \cdots \cap I(v_r)_{\lceil n\gamma a_r\rceil}
\subset I(v_1)_{a_1\lfloor n\gamma\rfloor}\cap \cdots\cap I(v_r)_{a_r\lfloor n\gamma\rfloor}
=\overline{m_S^{\lfloor n\gamma\rfloor}}
$$
for all $n\in \ZZ_{>0}$. Therefore 
\begin{equation}\label{eq91}
\ell(S/(J(v)_n)^m)\ge \ell(S/\overline{(m_S^{\lfloor n\gamma\rfloor}})^m)
\end{equation}
for all $m,n\in \ZZ_{>0}$. Thus
$$
\begin{array}{lll}
e(I(v)_n)
&\ge&\displaystyle\lim_{m\rightarrow\infty}d!\frac{\ell(S/(J(v)_n)^m)}{m^d}\\
&\ge& \displaystyle\lim_{m\rightarrow\infty}\frac{d!\ell(S/(\overline{m_S^{\lfloor n\gamma\rfloor}})^m)}{m^d}\\
&=&e(\overline{m_S^{\lfloor n\gamma\rfloor}})
=e(m_S^{\lfloor n\gamma\rfloor})=(\lfloor n\gamma\rfloor)^de(m_S)
\end{array}
$$
where the  first inequality is by (\ref{eq90}), the second inequality is by (\ref{eq91})
 and the second equality is by Proposition 11.2.1 \cite{HS} (and certainly by Rees's Theorem, Theorem \ref{ReesThm} of the introduction).
Thus
$$
e(\mathcal I(v))=\lim_{n\rightarrow\infty}\frac{e(I(v)_n)}{n^d}\ge \lim_{n\rightarrow\infty}\frac{(\lfloor n\gamma\rfloor)^d}{n^d}e(m_S)
=\gamma^de(m_S)>0.
$$
\end{proof}

\begin{Lemma}\label{LemmaLB} Suppose that $R$ is a $d$-dimensional Noetherian local ring and $\mathcal I$ is a graded $m_R$-family. Then $e(\mathcal I)=0$ implies $v(\mathcal I)=0$ for all $m_R$-valuations $v$.
\end{Lemma}

\begin{proof}
Suppose that $e(\mathcal I)=0$ and $v(\mathcal I)=c>0$ for some $m_R$-valuation $v$. Then
$I_n\subset I(v)_{\lceil cn\rceil}=I(v)^{(c)}_n$ for all $n$, so that  $e(\mathcal I)\ge c^d e(\mathcal I(v))>0$ by equation (\ref{eq80}), (\ref{eq50}) and  Lemma \ref{Lemma21}.
\end{proof}

The following theorem generalizes Proposition 3.12 \cite{BLQ} from analytically irreducible local rings to arbitrary Noetherian local rings.
We use the idea of their proof to give a proof which does not rely on the fact that volumes are limits on analytically irreducible local rings.

\begin{Theorem}\label{Theorem312} Let $R$ be a $d$-dimensional Noetherian local ring and $\mathcal I\subset \mathcal J$ be graded $m_R$-families. Suppose that 
$e(\mathcal I)=e(\mathcal J)$. Then 
$v(\mathcal I)= v(\mathcal J)$ for all $v\in V(R)$. 
\end{Theorem}

\begin{proof} We first prove the theorem when $R$ is  analytically irreducible. $\mathcal I\subset \mathcal J$ implies $v(\mathcal J)\le v(\mathcal I)$ for all $v\in V(R)$.
We assume that there exists $v\in V(R)$ such that $v(\mathcal J)<v(\mathcal I)$ and we will derive a contradiction. There exists $l\in \ZZ_{>0}$ such that $\frac{v(J_l)}{l}<v(\mathcal I)$. Thus there exists $f\in J_l\setminus I_l$ such that $v(f)<v(\mathcal I)l$. Let $v(f)=k$. For $m,n\in \ZZ_{\ge 0}$, there exist $R$-module homomorphisms
$$
\phi_{m,n}:R\rightarrow J_{lm}^n/I_{lm}^n,
$$
defined by $g\mapsto f^{mn}g$ for $g\in R$. If $g\in \mbox{Ker }\phi_{m,n}$, then 
$$
v(g)=v(f^{mn}g)-v(f^{mn})\ge v(I_{lm}^n)-mn v(f)\ge mnlv(\mathcal I)-mnk=mn(lv(\mathcal I)-k).
$$
Let $c=lv(\mathcal I)-k>0$. Then $\mbox{Ker }\phi_{m,n}\subset I(v)_{cmn}$ for all $m,n$. As a consequence, we have that
$$
\ell(J_{lm}^n/I_{lm}^n)\ge \ell(R/\mbox{Ker }\phi_{m,n})\ge \ell(R/I(v)_{cmn}).
$$
We have that $\ell(R/I_{lm}^n)-\ell(R/J_{lm}^n))=\ell(J_{lm}^n/I_{lm}^n)$, so
$$
e(I_{lm})-e(J_{lm})=\lim_{n\rightarrow \infty}\frac{d!\ell(J_{lm}^n/I_{lm}^n)}{n^d}.
$$
As shown in the proof of Lemma \ref{Lemma21}, there exists $\gamma\in \QQ_{>0}$ such that $I(v)_t\subset \overline {m_R^{\lfloor t\gamma\rfloor}}$ for all $t\in \ZZ_{>0}$. Thus $\ell(R/I(v)_t)\ge \ell(R/\overline {m_R^{\lfloor t\gamma\rfloor}})$ for all $t\in \ZZ_{>0}$, so that
$$
\ell(J_{lm}^n/I_{lm}^n)\ge \ell(R/\overline {m_R^{\lfloor c\gamma mn\rfloor}})
$$
for all $m,n$. Since $R$ is analytically unramified, there exists $h\in \ZZ_{\ge 0}$ such that $\overline{m_R^n}\subset m_R^{n-h}$ for all $n\in \NN$ by Theorem 1.4 \cite{R1}. Thus by Theorem 1.1 \cite{R2}, there exists a polynomial $h(x)\in \QQ[x]$ of degree $d$ such that $\ell(R/\overline{m^n})=h(n)$ for $n\gg 0$ and the leading coefficient of $h(x)$ is $\frac{e(m_R)}{d!}$. We have that 
$$
\lim_{n\rightarrow\infty}\frac{h(\lfloor c\gamma mn\rfloor)}{n^d}=\lim_{n\rightarrow\infty}\frac{e(m_R)}{d!}\frac{(\lfloor c\gamma mn\rfloor)^d}{n^d}
=\frac{e(m_R)(c\gamma m)^d}{d!}.
$$
Thus $e(I_{lm})-e(J_{lm})\ge (c\gamma m)^de(m_R)
$
for all $m\in \ZZ_{>0}$. We thus have that
$$
e(\mathcal I)-e(\mathcal J)=\lim_{m\rightarrow\infty}\frac{e(I_{lm})}{m^d}-\lim_{m\rightarrow\infty}\frac{e(J_{lm})}{m^d}
\ge \lim_{m\rightarrow\infty}\frac{(c\gamma m)^de(m_R)}{m^d}=(c\gamma)^de(m_R)>0,
$$
giving the desired contradiction,  establishing the theorem in the case that $R$ is analytically irreducible. 

Now suppose that $R$ is a noetherian local ring, with no other restrictions. 
Let $P_1,\ldots,P_r$ be the minimal primes of $\hat R$ such that $\dim \hat R/P=d$. 
$\mathcal I\subset \mathcal J$ implies $\mathcal I(\hat R/P_i)\subset \mathcal J(\hat R/P_i)$, so that $e(\mathcal J(\hat R/P_i))\le e(\mathcal I(\hat R/P_i)$ for all $i$ by (\ref{eq80}).
$$
e(\mathcal J)=e(\mathcal J\hat R)=\sum\ell_{\hat R_{P_i}}e(\mathcal J(\hat R/P_i))\le \sum\ell_{\hat R_{P_i}}e(\mathcal I(\hat R/P_i))=e(\mathcal I\hat R)=e(\mathcal I).
$$
Thus $e(\mathcal J)=e(\mathcal I)$ implies 
$e(\mathcal I(\hat R/P_i))=e(\mathcal J(\hat R/P_i))$ for all $i$, and so for all $i$, $v(\mathcal I(\hat R/P_i))=v(\mathcal J(\hat R/P_i))$ for all $v\in V(\hat R/P_i)$ by the first part of this proof, where the theorem is proven for analytically irreducible local rings. 
Thus $v(\mathcal I)=v(\mathcal J)$ for all $v\in V(R)$ by Lemma \ref{LemVal}.
\end{proof}

The following theorem generalizes Proposition 3.13 \cite{BLQ} from analytically irreducible local rings to arbitrary Noetherian local rings. The proof for the essential case when $R$ is analytically irreducible is from  \cite{BLQ}, which we reproduce for completeness. 

\begin{Theorem}\label{Theorem313} Let $R$ be a $d$-dimensional Noetherian local ring and $\mathcal I$, $\mathcal J$ be graded $m_R$-families. Suppose that $v(\mathcal J)\le v(\mathcal I)$ for all $v\in V(R)$. Then $e(\mathcal J)\le e(\mathcal I)$. 
\end{Theorem}

\begin{proof} We first prove the theorem in the case that $R$ is a complete local domain, so that it is excellent. We present  the  proof of Proposition 3.12 \cite{BLQ}.

There  exists a sequence of projective birational $R$-morphisms
$$
\cdots \rightarrow X_2\rightarrow X_1\rightarrow \mbox{Spec}(R)
$$
such that each $X_i$ is the normalization of the blowup of an $m_R$-primary ideal in $R$ such that $I_n\mathcal O_{X_n}$ and $J_n\mathcal O_{X_n}$ are invertible, with effective Cartier divisors $F_n$ and $G_n$ on $X_n$ such that $I_n\mathcal O_{X_n}=\mathcal O_{X_n}(-F_n)$ and $J_n\mathcal O_{X_n}=\mathcal O_{X_n}(-G_n)$. 
We will denote the $R$-morphisms in the above sequence by $\pi_{n,m}:X_{n}\rightarrow X_{m}$ for $n>m$. If $\lambda:Y\rightarrow X_n$ is a birational $R$-morphism, we will identify $\lambda^*(F_n)$ and $\lambda^*(G_n)$ with $F_n$ and $G_n$ respectively. 
The  prime Weil divisors of $X_n$ which contract to $m_R$ will be denoted by $\mbox{Exc}(X_n)$. Then as Weil divisors on $X_m$ with $m\ge n$,
$$
F_m=\sum_{E\in {\rm Exc}(X_m)}\frac{v_{E_i}(I_m)}{m}E\mbox{ and }G_m=\sum_{E\in {\rm Exc}(X_m)}\frac{v_{E_i}(J_m)}{m}E.
$$

Let $\epsilon>0$. By Theorems \ref{Theorem20} and  \ref{Theorem70}, there exists $m_0>0$ such that 
\begin{equation}\label{eqR1}
-\left(\left(\frac{-F_{m_0}}{m_0}\right)^d\right)=\frac{e(I_{m_0})}{m_0^d}<e(\mathcal I)+\epsilon.
\end{equation}
If $m_1$ is a multiple of $m_0$, then $\frac{-F_{m_0}}{m_0}\le \frac{-F_{m_1}}{m_1}$ since $I_{m_0}^{\frac{m_1}{m_0}}\subset I_{m_1}$,
 so
\begin{equation}\label{eqR2}
\left(\left(\frac{-F_{m_0}}{m_0}\right)^{d-1}\cdot \frac{F_{m_1}}{m_1}\right)\le \left(\left(\frac{-F_{m_0}}{m_0}\right)^{d-1}\cdot  \frac{F_{m_0}}{m_0}\right)
\end{equation}
 by Lemma \ref{Lemma4}, since $-\frac{F_{m_0}}{m_0}$ and $-\frac{F_{m_1}}{m_1}$ are nef.

\begin{equation}\label{eqRM}
\begin{array}{l}
-\left(\left(-\frac{F_{m_0}}{m_0}\right)^{d-1}\cdot \left(-\frac{F_{m_1}}{m_1}\right)\right)+\left(\left(-\frac{F_{m_0}}{m_0}\right)^{d-1}\cdot\left(\frac{-G_{m_1}}{m_1}\right)\right)\\
=\sum_{E\in{\rm Exc}(X_{m_1})}\left(\frac{v_E(I_{m_1})}{m_1}-\frac{v_E(J_{m_1})}{m_1}\right)\left(E\cdot\left(\frac{-F_{m_0}}{m_0}\right)^{d-1}\right)\\
=\sum_{E\in{\rm Exc}(X_{m_1})}\left(\frac{v_E(I_{m_1})}{m_1}-\frac{v_E(J_{m_1})}{m_1}\right)\left((\pi_{m_1,m_0})_*E\cdot\left(\frac{-F_{m_0}}{m_0}\right)^{d-1}\right)\\
= \sum_{E\in{\rm Exc}(X_{m_0})}\left(\frac{v_E(I_{m_1})}{m_1}-\frac{v_E(J_{m_1})}{m_1}\right)\left(E\cdot\left(\frac{-F_{m_0}}{m_0}\right)^{d-1}\right)
\end{array}
\end{equation}
where the second equality is by the projection formula, Proposition 2.8 \cite{CM}, and the third equality is since $(\pi_{m_1,m_0})_*E=0$ if $\dim \pi_{m_1,m_0}(E)<d-1$.
Given $\delta>0$, for $m_1$ sufficiently large which is divisible by $m_0$,
$\frac{v_E(J_{m_1})}{m_1}\le v(\mathcal J)+\delta$ for all $E\in \mbox{Exc}(X_{m_0})$. Thus 
$$
\frac{v_E(I_{m_1})}{m_1}-\frac{v(J_{m_1})}{m_1}\ge v(\mathcal I)-v(\mathcal J)-\delta\ge -\delta.
$$
We can choose $\delta$ sufficiently small  that 
$$
 \sum_{E\in{\rm Exc}(X_{m_0})}\left(\frac{v_E(I_{m_1})}{m_1}-\frac{v_E(J_{m_1})}{m_1}\right)\left(E\cdot\left(\frac{-F_{m_0}}{m_0}\right)^{d-1}\right)\ge -\epsilon.
 $$
Thus
$$
-\left(\left(\frac{-F_{m_0}}{m_0}\right)^{d-1}\cdot \left(\frac{-G_{m_1}}{m_1}\right)\right)\le-\left(\left(-\frac{F_{m_0}}{m_0}^{d-1}\right)\cdot \left(\frac{-F_{m_1}}{m_1}\right)\right)+\epsilon,
$$
so by (\ref{eqR1}) and (\ref{eqR2}),
$$
-\left(\left(-\frac{F_{m_0}}{m_0}\right)^{d-1}\cdot\left(\frac{-G_{m_1}}{m_1}\right)\right)\le e(\mathcal I)+2\epsilon.
$$
For any multiple $m_2$ of $m_1$, 
$$
-\left(\left(\frac{-F_{m_0}}{m_0}\right)^{d-2}\cdot\left(\frac{-G_{m_1}}{m_1}\right)\cdot \left(\frac{-F_{m_2}}{m_2}\right)\right)\le e(\mathcal I)+2\epsilon
$$
by Lemma \ref{Lemma4}, 
since $\frac{-F_{m_0}}{m_0}\le \frac{-F_{m_2}}{m_2}$   are nef. 

As in (\ref{eqRM}), we have that
$$
\begin{array}{l}
-\left(\left(\frac{-F_{m_0}}{m_0}\right)^{d-2}\cdot\left(\frac{-G_{m_1}}{m_1}\right)\cdot\left(\frac{-F_{m_2}}{m_2}\right)\right)
+\left(\left(\frac{-F_{m_0}}{m_0}\right)^{d-2}\cdot\left(\frac{-G_{m_1}}{m_1}\right)\cdot\left(\frac{-G_{m_2}}{m_2}\right)\right)\\
=\sum_{E\in{\rm Exc}(X_{m_2})}\left(\frac{v_E(I_{m_2})}{m_2}-\frac{v_E(J_{m_2})}{m_2}\right)\left(E\cdot \left(\frac{-F_{m_0}}{m_0}\right)^{d-2}\cdot \left(\frac{-G_{m_1}}{m_1}\right)\right)\\
=\sum_{E\in{\rm Exc}(X_{m_1})}\left(\frac{v_E(I_{m_2})}{m_2}-\frac{v_E(J_{m_2})}{m_2}\right)\left(E\cdot \left(\frac{-F_{m_0}}{m_0}\right)^{d-2}\cdot \left(\frac{-G_{m_1}}{m_1}\right)\right)\\
\end{array}
$$
and we may choose $m_2$ sufficiently large and divisible by $m_1$ so that 
$$
\begin{array}{l}
-\left(\left(\frac{-F_{m_0}}{m_0}\right)^{d-2}\cdot\left(\frac{-G_{m_1}}{m_1}\right)\cdot\left(\frac{-G_{m_2}}{m_2}\right)\right)\\
\le -\left(\left(\frac{-F_{m_0}}{m_0}\right)^{d-2}\cdot\left(\frac{-G_{m_1}}{m_1}\right)\cdot\left(\frac{-F_{m_2}}{m_2}\right)\right)+\epsilon.
\end{array}
$$
Thus 
$$
-\left(\left(\frac{-F_{m_0}}{m_0}\right)^{d-2}\cdot\left(\frac{-G_{m_1}}{m_1}\right)\cdot\left(\frac{-G_{m_2}}{m_2}\right)\right)\le e(\mathcal I)+3\epsilon.
$$
We obtain by induction that for sufficiently large $m_3,\ldots, m_d$ where $m_{i-1}$ divides $m_i$ for all $i$,
$$
-\left(\left(\frac{-G_{m_1}}{m_1}\right)\cdot \left(\frac{-G_{m_2}}{m_2}\right)\cdot\ldots\cdot \left(\frac{-G_{m_d}}{m_d}\right)\right)\le e(\mathcal I)+(d+1)\epsilon.
$$
By Theorem \ref{Theorem70} and Theorem \ref{Theorem20},
$$
e(\mathcal J)=\inf_{n_1,n_2,\ldots,n_d}-\left(\left(\frac{-G_{n_1}}{n_1}\right)\cdot \left(\frac{-G_{n_2}}{n_2}\right)\cdot\ldots\cdot \left(\frac{-G_{n_d}}{n_d}\right)\right).
$$
We obtain that $e(\mathcal J)<e(\mathcal I)+(d+1)\epsilon$ for all $\epsilon>0$. Thus $e(\mathcal J)\le e(\mathcal I)$, establishing the theorem for complete domains.

Now assume that $R$ is an arbitrary Noetherian local ring. 
Let $P_1,\ldots,P_r$ be the minimal primes of $\hat R$ such that $\dim \hat R/P_i=d$. $v(\mathcal J)\le v(\mathcal I)$ for all $v\in V(R)$ implies that $v(\mathcal J(\hat R/P_i))\le v(\mathcal I(\hat R/P_i))$ for all $v\in V(\hat R/P_i)$ and all $i$ by Lemma \ref{LemVal}. Thus $e(\mathcal J(\hat R/P_i))\le e(\mathcal I(\hat R/P_i))$ for all $i$ by the first part of the proof. Thus
$$
e(\mathcal J)=e(\mathcal J\hat R)=\sum\ell_{\hat R_{P_i}}(\hat R_{P_i})e(\mathcal J(\hat R/P_i))\le\sum\ell_{\hat R_{P_i}}(\hat R_{P_i})e(\mathcal I(\hat R/P_i))=e(\mathcal I\hat R)=e(\mathcal I).
$$
\end{proof}

We  now give the  proof of  Theorem \ref{Theorem14*}, stated in the introduction. For the readers convenience, we restate it here. It is proven for the essential case of analytically irreducible local rings in \cite{BLQ}.

\begin{Theorem}(Theorem \ref{Theorem14*}\  of the introduction) \label{Theorem14} Suppose that $R$ is a Noetherian local ring and $\mathcal I\subset \mathcal J$ are graded $m_R$-families. Then $e(\mathcal I)=e(\mathcal J)$ if and only if $\tilde{\mathcal I}=\tilde{\mathcal J}$. In particular, $\tilde{\mathcal I}$ is the largest $m_R$-filtration $\mathcal J$ such that $\mathcal I\subset \mathcal J$ and $e(\mathcal J)=e(\mathcal I)$.
\end{Theorem}

\begin{proof}
If  $\tilde {\mathcal J}=\tilde{\mathcal I}$ then $v(\mathcal I)=v(\mathcal J)$ for all $v\in V(R)$ which implies $e(\mathcal I)= e(\mathcal J)$ by Theorem \ref{Theorem313}. 

If $e(\mathcal J)=e(\mathcal I)$, then $v(\mathcal I)=v(\mathcal J)$ for all $v\in V(R)$ by Theorem \ref{Theorem312} which implies $\tilde{\mathcal I}=\tilde{\mathcal J}$. \end{proof}

\begin{Corollary}(Rees \cite{R}, Theorem 11.3.1 \cite{HS}, Theorem \ref{ReesThm} of Subsection \ref{ReesSubSec} of the introduction)\label{ReesThmpf} Suppose that $R$ is a formally equidimensional $d$-dimensional local ring and $I\subset J$ are $m_R$-primary ideals. Then $J\subset \overline{I}$ if and only if $e(J)=e(I)$.
\end{Corollary}

\begin{proof}  Let $\mathcal I=\{I^n\}$ and $\mathcal J=\{J^n\}$. Then $e(\mathcal I)=e(I)$ and $e(\mathcal J)=e(J)$.
For $v\in V(R)$, $v(\mathcal I)=v(I)$ and $v(\mathcal J)=v(J)$. Thus $e(J)=e(I)$ if and only if $v(I)=v(J)$ for all $v\in V(R)$ by Theorem  \ref{Theorem14} , which is equivalent to $v(I)=v(J)$ for all Rees valuations $v$ of $I$, since all Rees valuations of $R$ are in $V(R)$ as $R$ is formally equidimensional.
This is equivalent to $J\subset \overline I$.
\end{proof}

Theorem \ref{Theorem13} is a consequence of the following Corollary.
\begin{Corollary}\label{ReesInt} Suppose that $R$ is a $d$-dimensional Noetherian local ring, $\mathcal I\subset \mathcal J$ are graded $m_R$-families and $\oplus_{n\ge 0}  J_n$ is integral over $\oplus_{n\ge 0}  I_n$. Then $e(\mathcal J)=e(\mathcal I)$.
\end{Corollary}

\begin{proof} $\mathcal J\subset \tilde{\mathcal I}$ by Lemma \ref{SatInt}.
\end{proof}

\begin{Corollary}\label{ReesDiv} Suppose that $\mathcal I\subset \mathcal J$, $e(\mathcal I)=e(\mathcal J)$ and $\tilde{\mathcal I}=\overline{\mathcal I}$. Then $\oplus_{n\ge 0}J_n$ is integral over $\oplus_{n\ge 0}I_n$.
\end{Corollary}

Theorem 1.2 of \cite{C8} (Theorem \ref{NewRBT} of SubSection \ref{ReesSubSec} of the Introduction) follows from this corollary and the following lemma.

\begin{Lemma} \label{DivSat}
Suppose that $\mathcal I$ is a divisorial filtration. Then
 $\mathcal I=\tilde{\mathcal I}$.
\end{Lemma}

\begin{proof}
Let $\mathcal I=\{I_n\}$. 
There exists $v_1,\ldots,v_r\in V(R)$ and $a_1,\ldots, a_r\in \RR_{>0}$ such that 
$$
I_n=\{f\in R\mid v_i(f)\ge na_i\mbox{ for }1\le i\le r\}.
$$ 
Now $\frac{v_i(I_n)}{n}\ge a_i$ for all $i$ and $n$ and $\inf \frac{v_i(I_n)}{n}=v_i(\mathcal I)$ implies $a_i\le v_i(\mathcal I)$ for all $i$. Thus
$$
\{f\in R\mid v_i(f)\ge nv_i(\mathcal I)\mbox{ for }1\le i\le r\}\subset I_n.
$$
Therefore
$$
I_n\subset \tilde I_n =\cap_{v\in V(R)}\{f\in R\mid v(f)\ge nv(\mathcal I)\}
\subset \{f\in R\mid v_i(f)\ge nv_i(\mathcal I)\mbox{ for }1\le i\le r\}\subset I_n
$$
for all $n$, and so $\tilde{\mathcal I}=\mathcal I$. 
\end{proof}

\section{Minkowski inequality and equality}\label{SecMink} In this section we prove very general Minkowski inequalities and the characterization of equality, for graded $m_R$-families, generalizing classical results for ideals by Teissier \cite{T1}, \cite{T2}, \cite{T3},  Rees and Sharp \cite{RS} and Katz \cite{Ka}. Some history and discussion of Minkowski inequalities for ideals and graded families is given in Subsection \ref{SubsecMink} of the introduction.

\begin{Theorem}\label{MinkIneq}(Minkowski Inequalities)  Suppose that $R$ is a Noetherian $d$-dimensional  local ring and  $\mathcal I=\{I_n\}$ and $\mathcal J=\{J_n\}$ are graded $m_R$-families. Let $e_i=e(\mathcal I^{[d-i]},\mathcal J^{[i]})$ for $0\le i\le d$.
Then 
\begin{enumerate}
\item[1)] $e_i^2\le e_{i-1}e_{i+1}$  for $1\le i\le d-1$,
\item[2)] $e_ie_{d-i}\le e_0e_d$ for $0\le i\le d$,
\item[3)]  $e_i^d\le e_0^{d-i}e_d^i$ for $0\le i\le d$,
\item[4)] $e(\mathcal I\mathcal J)^{\frac{1}{d}}\le e_0^{\frac{1}{d}}+e_d^{\frac{1}{d}}$
where $\mathcal I\mathcal J=\{I_nJ_n\}$.
\end{enumerate}
\end{Theorem}

\begin{proof} All of the inequalities of the theorem hold between the ideals $I_n$ and $J_n$ for all $n\in \ZZ_{>0}$, with respect to ordinary multiplicities and mixed multiplicities of $m_R$-primary ideals by Theorem 17.7.2 and Corollary 17.7.3 \cite{HS}.

In particular, 
$$
e(I_nJ_n)^{\frac{1}{d}}\le e(I_n)^{\frac{1}{d}}+e(J_n)^{\frac{1}{d}}
$$
for all $n$. Thus
$$
\left(\frac{e(I_nJ_n)}{n^d}\right)^{\frac{1}{d}}\le \left(\frac{e(I_n)}{n^d}\right)^{\frac{1}{d}}+\left(\frac{e(J_n)}{n^d}\right)^{\frac{1}{d}}
$$
for all $n$, and taking the limit as $n$ goes to infinity, we obtain by Theorem \ref{multlim} that
$$
\left(\lim_{n\rightarrow\infty}\frac{e(I_nJ_n)}{n^d}\right)^{\frac{1}{d}}\le \left(\lim_{n\rightarrow\infty}\frac{e(I_n)}{n^d}\right)^{\frac{1}{d}}+\left(\lim_{n\rightarrow\infty}\frac{e(J_n)}{n^d}\right)^{\frac{1}{d}},
$$
giving the  Minkowski inequality 4).  

By (\ref{eqIn15}), 
$$
e_i=e(\mathcal I^{[d-i]},\mathcal J^{[i]})=\lim_{n\rightarrow\infty} \frac{e(I_n^{[d-i]},J_n^{[i]})}{n^d}
$$
and the first three inequalities then follow from  the Minkowski inequalities for $m_R$-primary ideals.
\end{proof}

\begin{Remark}\label{RemMink} Suppose that $d=\dim R\ge 2$ and $\mathcal I$ and $\mathcal J$ are graded $m_R$-families with $e(\mathcal I)>0$ and $e(\mathcal J)>0$. If equality holds in one of the Minkowski inequalities of Theorem \ref{MinkIneq}, then equality holds in all four of the Minkowski inequalities.
\end{Remark}

The remark follows (for instance) from the analysis of Section 9 \cite{C8}.

The following theorem is proven when $R$ is analytically irreducible or an  excellent local domain in Theorem 10.3 \cite{C8} (stated in Theorem \ref{BrunnMink} of the introduction), using the theory of Okounkov bodies and the Brunn-Minkowski theorem of convex geometry. Our proof of Theorem \ref{Mink} is by reduction to the assumptions of Theorem \ref{BrunnMink}.

\begin{Theorem}\label{Mink}(Theorem \ref{Mink*} of the introduction) Suppose that $R$ is a $d$-dimensional Noetherian local ring with $d\ge 2$, and $\mathcal I$, $\mathcal J$ are graded $m_R$-families. Then the Minkowski equality 
$$
e(\mathcal I\mathcal J)^{\frac{1}{d}}=e(\mathcal I)^{\frac{1}{d}}+e(\mathcal J)^{\frac{1}{d}}
$$
holds if and only if 
$$
e(\mathcal J)^{\frac{1}{d}}v(\mathcal I)=e(\mathcal I)^{\frac{1}{d}}v(\mathcal J)
$$
for all $v\in V(R)$.
\end{Theorem}

\begin{proof} First suppose that $e(\mathcal J)^{\frac{1}{d}}v(\mathcal I)=e(\mathcal I)^{\frac{1}{d}}v(\mathcal J)$ for all $v\in V(R)$. If $e(\mathcal I)=0$ then $v(\mathcal I)=0$ for all $v\in V(R)$ by Lemma \ref{LemmaLB} so $v(\mathcal I\mathcal J)=v(\mathcal J)$ for all $v\in V(R)$ by Lemma \ref{LemmaLB} so $v(\mathcal I\mathcal J)=v(\mathcal J)$ for all $v\in V(R)$. Thus $\widetilde{\mathcal I\mathcal J)}=\tilde{\mathcal J}$ and so $e(\mathcal I\mathcal J)=e(\mathcal J)$ by Theorem \ref{Theorem14}. Thus $e(\mathcal I\mathcal J)^{\frac{1}{d}}=e(\mathcal I)^{\frac{1}{d}}+e(\mathcal J)^{\frac{1}{d}}$. Therefore we may assume that $e(\mathcal I)>0$ and $e(\mathcal J)>0$. thus $\tilde{\mathcal I}=\widetilde{J^{(c)}}$ where $c=\frac{e(\mathcal I)^{\frac{1}{d}}}{e(\mathcal J)}^{\frac{1}{d}}$.
Then 
$$
e(\mathcal I)=e(\tilde{\mathcal I})=e(\widetilde{\mathcal J^{(c)}})=e(\mathcal J^{(c)})=c^de(\mathcal J)
$$
by Theorem \ref{Theorem14} and (\ref{eq50}). For all $v\in V(R)$, we have that
$$
v(\mathcal I)=v(\mathcal J^{(c)})=cv(\mathcal J)
$$ 
by (\ref{eq51}).
Thus
$$
v(\{I_nJ_n\})=(c+1)v(\mathcal J)=v(\mathcal J^{(c+1)})\mbox{ for all }v\in V(R),
$$
so
$$
e(\widetilde{\{I_nJ_n\}}=e(\mathcal J^{(c+1)})=(c+1)^de(\mathcal J)
$$
by Theorem \ref{Theorem313}. It follows that
$$
e(\{I_nJ_n\})^{\frac{1}{d}}= e(\mathcal I)^{\frac{1}{d}}+e(\mathcal J)^{\frac{1}{d}}.
$$

Now suppose that $e(\mathcal J)^{\frac{1}{d}}v(\mathcal I)=e(\mathcal I)^{\frac{1}{d}}v(\mathcal J)$ for all $v\in V(R)$.
If $e(\mathcal I)=0$ (or $e(\mathcal J)=0$) then $v(\mathcal I)=0$ (or $v(\mathcal J)=0$) for all $v\in V(R)$ be Lemma \ref{LemmaLB}, so the conclusions of the theorem hold. Thus we may assume that $e(\mathcal I)>0$ and $e(\mathcal J)>0$.

Let $e_j=e(\mathcal I^{[d-j]},\mathcal J^{[j]})$ for $0\le i\le d$.
Let $P_i$, for $1\le i\le r$  be the minimal primes of $\hat R$ such that $\dim \hat R/P_i=d$ and let $e_j(i)=e((\mathcal I\hat R/P_i)^{[d-j]},(\mathcal J \hat R/P_i)^{[j]})$.
By Theorem \ref{Minkcoef*}, for $n_1,n_2\in \ZZ_{\ge 0}$, we have polynomial functions 
$$
P(n_1,n_2)=\lim_{m\rightarrow\infty}\frac{e(I_{mn_1}J_{mn_2})}{m^d}=\sum_{j=0}^d\frac{1}{(d-j)!j!}e_jn_1^{d-j}n_2^j,
$$
$$
P_i(n_1,n_2)=\lim_{m\rightarrow\infty}\frac{e(I_{mn_1}J_{mn_2}\hat R/P_i)}{m^d}=\sum_{j=0}^d\frac{1}{(d-j)!j!}e_j(i)n_1^{d-j}n_2^j.
$$
By Corollary \ref{Cor73}, $P(n_1,n_2)=\sum a_iP_i(n_1,n_2)$ where $a_i=\ell_{\hat R/P_i}(\hat R/P_i)$.

Since $e(\mathcal I)>0$, $e(\mathcal J)>0$ and the Minkowski equality $e(\mathcal I\mathcal J)^{\frac{1}{d}}=e(\mathcal I)^{\frac{1}{d}}+e(\mathcal J)^{\frac{1}{d}}$  is satisfied, by Remark \ref{RemMink},
 
\begin{equation}\label{eq64}
e_i^d=e_0^{d-i}e_d^i\mbox{ for }0\le i\le d
\end{equation}
and
\begin{equation}\label{eq65}
e_i^2=e_{i-1}e_{i+1}\mbox{ for all }i.
\end{equation}
Equation (\ref{eq64}) implies that 
$e_i>0$ for all $i$. Thus there exists $\xi\in \RR_{>0}$ such that
\begin{equation}\label{eq79}
\xi=\frac{e_1}{e_0}=\frac{e_2}{e_1}=\cdots=\frac{e_d}{e_{d-1}}.
\end{equation}
For fixed $j$ with $1<j<d$, we have that
\begin{equation}\label{eq41}
\begin{array}{lll}
0&\le &\sum_{i=1}^ra_i(e(i)_{j+1}^{\frac{1}{2}}-\xi e(i)_{j-1}^{\frac{1}{2}})^2\\
&=&\sum_{i=1}^ra_i(e(i)_{j+1}-2\xi e(i)_{j+1}^{\frac{1}{2}} e(i)_{j-1}^{\frac{1}{2}}+\xi^2 e(i)_{j-1})\\
&\le& \sum_{i=1}^ra_i(e(i)_{j+1}-2\xi e(i)_j+\xi^2 e(i)_{j-1})\\
&=& e_{j+1}-2\xi e_j+\xi^2e_{j-1}=0
\end{array}
\end{equation}
where the second inequality is by the first Minkowski inequality and the final equality with 0 is by (\ref{eq79}).

Thus
\begin{equation}\label{eq40}
e(i)_{j+1}^{\frac{1}{2}}=\xi e(i)_{j-1}^{\frac{1}{2}}\mbox{ for all $i$ and $j$}
\end{equation}
and
\begin{equation}\label{eq65'}
e(i)_j^2=e(i)_{j-1}e(i)_{j+1}
\end{equation}
for all $i$ and $j$ since the first and second inequalities in (\ref{eq41}) are equalities.

Thus for fixed $i$, either $e(i)_j=0$ for all $j$ or $e(i)_j>0$ for all $j$. 

We first consider the case where $i$ satisfies
\begin{equation}\label{eq81}
e(i)_j>0\mbox{ for all $j$.}
\end{equation}
Then equation (\ref{eq65'}) implies that there exists $\lambda_i\in \RR_{>0}$ such that
$$
\frac{e(i)_{j+1}}{e(i)_j}=\lambda_i\mbox{ for all $j$.}
$$
 Equation (\ref{eq40}) implies
$$
\xi^2=\frac{e(i)_{j+1}}{e(i)_{j-1}}=\frac{e(i)_{j+1}}{e(i)_j}\frac{e(i)_j}{e(i)_{j-1}}=\lambda_i^2
$$
for all $i$ and $j$, so that $\lambda_i=\xi$ and thus 
\begin{equation}\label{eq42}
\frac{e(i)_d^{\frac{1}{d}}}{e(i)_0^{\frac{1}{d}}}=\xi,
\end{equation}
since
$$
\frac{e(i)_d}{e(i)_0}=\frac{e(i)_d}{e(i)_{d-1}}\frac{e(i)_{d-1}}{e(i)_{d-2}}\cdots \frac{e(i)_1}{e(i)_0}=\xi^d.
$$

Since (\ref{eq65'}) holds and  $e(i)_0>0$ and $e(i)_d>0$, Remark \ref{RemMink}
 shows that 
the Minkowski equality 
$$
e(\mathcal I\mathcal J\hat R/P_i)^{\frac{1}{d}}=e(\mathcal I\hat R/P_i)^{\frac{1}{d}}+e(\mathcal J\hat R/P_i)^{\frac{1}{d}}
$$
holds   so we have that 
$$
v(\mathcal J(\hat R/P_i))=\frac{e(\mathcal J(\hat R/P_i))^{\frac{1}{d}}}{e(\mathcal I(\hat R/P_i)^{\frac{1}{d}}} v(\mathcal I(\hat R/P_i))
=\frac{e(i)_d^{\frac{1}{d}}}{e(i)_0^{\frac{1}{d}}}v(\mathcal I(\hat R/P_i))
$$
for all $v \in V(\hat R/P_i)$ by Theorem \ref{BrunnMink}. By (\ref{eq42}), 
\begin{equation}\label{eq43}
v(\mathcal J(\hat R/P_i))=\xi v(\mathcal I(\hat R/P_i))
\end{equation}
for all $v\in V(\hat R/P_i)$. Thus
$$
\frac{e(\mathcal J)^{\frac{1}{d}}}{e(\mathcal I)^{\frac{1}{d}}}v(\mathcal I)= \xi v(\mathcal I\hat R/P_i)=v(\mathcal J\hat R/P_i)=v(\mathcal J)
$$
by (\ref{eq79}) and (\ref{eq43}).

We now suppose that $i$ is such that the remaining case holds,
\begin{equation}\label{eq80*}
e(i)_j=0\mbox{ for all }j.
\end{equation}
$e(i)_0=e(i)_d=0$ implies $v(\mathcal I(\hat R/P_i))=0$ and $v(\mathcal J\hat R/P_i)=0$ for all $v\in V(\hat R/P_i)$ by Lemma \ref{LemmaLB}. Thus $ v(\mathcal J)=v(\mathcal I)=0$  so that 
$$
e(\mathcal J)^{\frac{1}{d}}v(\mathcal I)=e(\mathcal I)^{\frac{1}{d}}v(\mathcal J).
$$

\end{proof}

We deduce the Teissier-Rees-Sharp-Katz Theorem, which was discussed in Subsection \ref{SubsecMink} of the introduction, from Theorem \ref{Mink}. 

\begin{Theorem}\label{TRSK}(Teissier, Rees and  Sharp, Katz) Suppose that $R$ is a formally equidimensional Noetherian local ring of dimension $d\ge 2$ and $I,J$ are $m_R$-primary ideals. If 
$$
e(IJ)^{\frac{1}{d}}=e(I)^{\frac{1}{d}}+e(J)^{\frac{1}{d}},
$$
then there exist positive integers $a$ and $b$ such that
$$
\overline{I^a}=\overline{J^b}.
$$
\end{Theorem} 

\begin{proof} Let $e_i=e(I^{[d-i]},J^{[i]})$ for $0\le i\le d$. As shown in the proof of Theorem \ref{Mink}, all $e_i$ are positive, since $e_0$ and $e_d$ are, and by (\ref{eq41}) of the proof of Theorem \ref{Mink}, 
$$
\frac{e_1}{e_0}=\frac{e_2}{e_1}=\cdots \frac{e_d}{e_{d-1}},
$$
and so
$$
\left(\frac{e(J)}{e(I)}\right)^{\frac{1}{d}}=\frac{e_1}{e_0}.
$$
Thus
$$
e_1v(I)=e_0v(J)
$$
for all $v\in V(R)$, by Theorem \ref{Mink}, applied to the adic filtrations $\mathcal I=\{I^n\}$ and $\mathcal J=\{J^n\}$.
Since $R$ is formally equidimensional, all Rees valuations of $I$ and $J$ are in $V(R)$. Thus 
$$
\overline{I^{e_1}}=\overline{J^{e_0}}.
$$
(With the formally equidimensional assumption, we have that the saturations $\tilde{\mathcal I}=\{\overline{I^n}\}$ and $\tilde{\mathcal J}=\{\overline{J^n}\}$, as explained in Subsection \ref{SubSecVal} of the introduction.)
\end{proof}

The following corollary is essentially a rephrasing of Theorem \ref{Mink}. It is deduced in Corollary \cite{BLQ} when $R$ is analytically irreducible, using Theorem 10.3 \cite{C8} (Theorem \ref{BrunnMink} of the introduction).

\begin{Corollary}\label{Minkgen}(Corollary \ref{Minkgen*} of the Introduction) Suppose that $R$ is a $d$-dimensional Noetherian local ring with $d\ge 2$ and $\mathcal I,\mathcal J$ are graded $m_R$-families such that 
$e(\mathcal J)> 0$. 
Then there is equality in the Minkowski inequality
$$
e(\mathcal I\mathcal J)^{\frac{1}{d}}= e(\mathcal I)^{\frac{1}{d}}+e(\mathcal J)^{\frac{1}{d}}
$$
if and only if 
$$
\tilde{\mathcal I}=\widetilde{\mathcal J^{(c)}}\mbox{ for some $c\in \RR_{\ge 0}$}.
$$
\end{Corollary}

\begin{proof}  
If $\tilde{\mathcal I}=\widetilde{\mathcal J^{(c)}}$ then $e(\mathcal I\mathcal J)^{\frac{1}{d}}=e(\mathcal I)^{\frac{1}{d}}+e(\mathcal J)^{\frac{1}{d}}$ by Theorem \ref{Mink}.


We now prove that equality in the Minkowski inequality implies $\tilde{\mathcal I}=\tilde{\mathcal J^{(c)}}$ for some $c\in \RR_{\ge 0}$. 
Let $c=\frac{e(\mathcal I)^{\frac{1}{d}}}{e(\mathcal J)^{\frac{1}{d}}}$. By Theorem \ref{Mink}, we have that 
$$
v(\mathcal I)=c v(\mathcal J) \mbox{ for all $v\in V(R)$}.
$$
Thus, for $v\in V(R)$, we have that
$v(\mathcal I)=c v(\mathcal J)=v(\mathcal J^{(c)})$ by (\ref{eq51}). Thus $\tilde{\mathcal I}=\widetilde{\mathcal J^{(c)}}$.
\end{proof}

\begin{Corollary}\label{BMink} Suppose that $R$ is a $d$-dimensional Noetherian local ring with $d\ge 2$, $\mathcal I$ and $\mathcal J$ are divisorial filtrations.
Then there is equality in the Minkowski inequality,
$$
e(\{I_nJ_n\})^{\frac{1}{d}}= e(\mathcal I)^{\frac{1}{d}}+e(\mathcal J)^{\frac{1}{d}}
$$
if and only if 
$$
\mathcal I=\mathcal J^{(c)}\mbox{ for some $c\in \RR_{>0}$}
$$
\end{Corollary}

\begin{proof} Since $\mathcal I$ is divisorial, $\mathcal I=\tilde {\mathcal I}$ by Lemma \ref{DivSat}. Similarly, since $\mathcal J$ is divisorial,  
$\mathcal J^{(c)}$ is divisorial and 
$\mathcal J^{(c)}=\widetilde{\mathcal J^{(c)}}$. We have that $e(\mathcal I)>0$ and $e(\mathcal J)>0$ by Lemma \ref{LemmaLB}. 
\end{proof}

\section{The Minkowski equality for bounded filtrations}\label{SecMinkbound} 
In this section we prove Theorem \ref{BoundMink}. We continue to assume that $R$ is a Noetherian local ring. 

Suppose that $\mathcal I=\{I_n\}$ is a divisorial $m_R$-filtration, so that there exist $v_1,\ldots,v_r\in V(R)$ and $a_1,\ldots,a_r\in \RR_{\ge 0}$ such that 
$$
I_n=\{f\in R\mid v_i(f)\ge a_i n\mbox{ for }1\le i\le r\}.
$$
Then
$$
\tilde I_n=\{f\in R\mid v_i(f)\ge nv_i(\mathcal I)\mbox{ for }1\le i\le r\}
$$
by the proof of Lemma \ref{DivSat}. 

Now suppose that $\mathcal I=\{I_n\}$ is graded $m_R$-family. As in equation (74) of \cite{C8}, we define
$$
w_{\mathcal I}(f)=\max\{m\mid f\in I_m\}
$$
for $f\in R$. $w_{\mathcal I}(f)$ is either a natural number or $\infty$.

The following lemma is proven for less general rings in Lemma 11.3 \cite{C8}. The proof is the same.

\begin{Lemma}\label{ratfun} Suppose that $R$ is a Noetherian local ring and $\mathcal I$ is a $\QQ$-divisorial $m_R$-filtration. Suppose that $f\in m_R\setminus Q_1\cup\cdots \cup Q_r$ where $Q_1,\ldots,Q_r$ are the minimal primes of $R$ such that $v_i\in V(R/Q_i)$. Then $w_{\mathcal I}(f)<\infty$ and there exists $s\in \ZZ_{>0}$ such that 
$$
w_{\mathcal I}(f^{ns})=nw_{\mathcal I}(f^s)
$$
for all $n\in \ZZ_{>0}$.
\end{Lemma}

The following theorem generalizes Theorem 11.4 of \cite{C8} which assumes that $R$ is a $d$-dimensional normal local ring. The proof in the general case is essentially the same, using the more general Theorem \ref{Mink} instead of Theorem 10.3 \cite{C8}.

\begin{Theorem}\label{ValThm} Suppose that $R$ is a $d$-dimensional Noetherian local ring with $d\ge 2$
and $\mathcal I$ and $\mathcal J$ are $\QQ$-divisorial $m_R$-filtrations. 
Suppose that the Minkowski equality 
$$
e(\mathcal I\mathcal J)^{\frac{1}{d}}=e(\mathcal I)^{\frac{1}{d}}+e(\mathcal J)^{\frac{1}{d}}
$$
holds between $\mathcal I$ and $\mathcal J$. Then
$$
\xi=\frac{e(\mathcal J)^{\frac{1}{d}}}{e(\mathcal I)^{\frac{1}{d}}}\in \QQ_{>0}.
$$
\end{Theorem}

\begin{proof} The multiplicities $e(\mathcal I)>0$ and $e(\mathcal J)>0$ since $\mathcal I$ and $\mathcal J$ are divisorial filtrations by Lemma \ref{Lemma21}. There exist $v_1,\ldots,v_r\in V(R)$ and $a_1,\ldots,a_r,b_1,\ldots,b_r\in \QQ_{\ge 0}$ such that
$\mathcal I=\{I_n\}$ where
$$
I_n=\{f\in R\mid v_i(f)\ge na_i\mbox{ for }1\le i\le r\}
$$
and $\mathcal J=\{J_n\}$ where 
$$
J_n=\{f\in R\mid f_i(f)\ge nb_i\mbox{ for }1\le i\le r\}.
$$
We have that 
$$
v_i(\mathcal I)e(\mathcal J)^{\frac{1}{d}}=v_i(\mathcal J)e(\mathcal I)^{\frac{1}{d}}\mbox{ for}1\le i\le r
$$
by Theorem \ref{Mink}, since the Minkowski equality holds. 

Let $g\in m_R\setminus Q_1\cup \cdots \cup Q_r$ where $v_i\in V(R/Q_i)$. By Lemma \ref{ratfun}, there exists $s\in \ZZ_{>0}$ such that 
$$
w_{\mathcal I}(g^{sl})=lw_{\mathcal I}(g^s)\mbox{ and } w_{\mathcal J}(g^{sl})=lw_{\mathcal J}(g^s)
$$
for all $l>0$. Let $g=g^s$. Let $m=w_{\mathcal J}(f)\in \ZZ_{>0}$ and $n=w_{\mathcal I}(f)\in \ZZ_{>0}$. Let $\delta\in \RR_{>0}$. By the remark on the bottom of page 156 \cite{HW} (or Lemma 4.1 \cite{C8}) there exists $\alpha\in \RR_{>0}$ such that $\alpha<\frac{\delta}{m}$ and there exist positive integers $p_0,q_0,p_0',q_0'$ such that
$$
\xi-\frac{\alpha}{q}<\frac{p_0}{q_0}\le \xi
$$
and
$$
\xi\le \frac{p_0'}{q_0'}<\xi+\frac{\alpha}{q_0'}.
$$
Let $p=p_0m$, $q=q_0m$. Then 
$$
v_i(\mathcal I)p\le \xi v_i(\mathcal I)q=v_i(\mathcal J)q
$$
for all $i$ so that $J_q\subset I_p$. $f^{q_0}\in J_q\subset I_p$ so that $w_{\mathcal I}(f^{q_0})\ge p$. Thus since $w_{\mathcal I}(f^{q_0})=q_0w_{\mathcal I}(f)$ by our choice of $f$,
\begin{equation}\label{Val3}
w_{\mathcal I}(f)\ge \frac{p_0}{q_0}m=\frac{p_0}{q_0}w_{\mathcal J}(f)>(\xi-\frac{\alpha}{q_0})w_{\mathcal J}(f).
\end{equation}
Let $p'=p_0'n$, $q'=q_0'n$. Then 
$$
v_i(\mathcal I)\xi q'\le v_i(\mathcal I)p'
$$
for all $i$ so that $I_{p'}\subset J_{q'}$. $f^{p_0'}\in I_{p'}\subset I_{q'}$. Thus since $w_{\mathcal J}(f^{p_0'})=p_0'w_{\mathcal J}(f)$ by our choice of $f$, 
$$
w_{\mathcal J}(f)\ge \frac{q_0'n}{p_0'}=\frac{q_0'}{p_0'}w_{\mathcal I}(f),
$$
so
\begin{equation}\label{Val4}
w_{\mathcal I}(f)\le \frac{p_0'}{q_0'}w_{\mathcal J}(f)<(\xi+\frac{\alpha}{q_0'})w_{\mathcal J}(f).
\end{equation}
Combining equations (\ref{Val3}) and (\ref{Val4}), we have that
$$
\begin{array}{lll}
w_{\mathcal I}(f)&\le & (\xi+\frac{\alpha}{q_0'})w_{\mathcal J}(f)\\
&=& (\xi-\frac{\alpha}{q_0})w_{\mathcal J}(f)+(\frac{\alpha}{q_0}+\frac{\alpha}{q_0'})w_{\mathcal J}(f)\\
&<&w_{\mathcal I}(f)+2\delta.
\end{array}
$$
All these inequalities approach equality when the limit is taken as $\delta$ goes to zero. Thus $w_{\mathcal I}(f)=\xi w_{\mathcal J}(f)$, and so
$$
\xi=\frac{w_{\mathcal I}(f)}{w_{\mathcal J}(f)}\in \QQ_{>0}.
$$
\end{proof}

\begin{Remark} The multiplicity of a $\QQ$-divisorial filtration can be an irrational number (Example 6 \cite{CS} and Examples in \cite{C7}).
\end{Remark}

The following theorem is Theorem \ref{BoundMink} of the introduction.

\begin{Theorem}\label{BoundMink*}(Proven in Theorem 1.4 \cite{C8} when $R$ is analytically irreducible  or an excellent local domain)
  Suppose that $R$ is a $d$-dimensional Noetherian  local ring of dimension $d\ge 2$ and $\mathcal I$ and $\mathcal J$ are $\QQ$-bounded $m_R$-filtrations.
Then the  Minkowski equality
\begin{equation}\label{Val11}
e(\mathcal I\mathcal J)^{\frac{1}{d}}=e(\mathcal I)^{\frac{1}{d}}+e(\mathcal J)^{\frac{1}{d}}
\end{equation}
holds if and only if there exist positive integers a,b such that there is equality of integral closures
$\overline{\mathcal I^{(a)}}=\overline{\mathcal J^{(b)}}$.
\end{Theorem}

\begin{proof} Suppose that the Minkowski equality (\ref{Val11}) holds. Then
$$
e(\mathcal I\mathcal J)^{\frac{1}{d}}=e(\mathcal I)^{\frac{1}{d}}+e(\mathcal J)^{\frac{1}{d}}
$$
so the $\QQ$-divisorial $m_R$-filtrations $\tilde {\mathcal I}$ and $\tilde {\mathcal J}$ satisfy the Minkowski equality. Thus 
$$
\frac{e(\mathcal J)^{\frac{1}{d}}}{e(\mathcal I)^{\frac{1}{d}}}=\frac{e(\tilde{\mathcal J})^{\frac{1}{d}}}{e(\tilde{\mathcal I})^{\frac{1}{d}}}\in \QQ_{>0}
$$
by Theorem \ref{ValThm}. Write 
$$
\frac{e(\tilde{\mathcal J})^{\frac{1}{d}}}{e(\tilde{\mathcal I})^{\frac{1}{d}}}=\frac{a}{b}
$$
with $a,b\in \ZZ_{>0}$. Then $\widetilde{\mathcal I^{(a)}}=\widetilde{\mathcal J^{(b)}}$ by Theorem \ref{Mink}. Thus $\overline{\mathcal I^{(a)}}=\overline{\mathcal J^{(b)}}$ since $\mathcal I$ and $\mathcal J$ are bounded $m_R$-filtrations.

Now suppose that $\overline{\mathcal I^{(a)}}=\overline{\mathcal J^{(b)}}$ for some positive integers $a$ and $b$. Then $\widetilde{\mathcal I^{(a)}}=\widetilde{\mathcal J^{(b)}}$. Thus the Minkowski equality (\ref{Val11}) holds by Theorem \ref{Mink}.

\end{proof}

\end{document}